\documentclass[12pt]{amsart}
\usepackage{amssymb,amsfonts,graphicx,amsmath,amsthm,amstext,latexsym,amsfonts}
\usepackage{graphicx,epsfig}
\usepackage{bbm}
\usepackage{t1enc}
\usepackage{subfig}
\usepackage{hyperref}

\usepackage{a4wide}
\usepackage{tikz}
\usetikzlibrary{intersections}
\theoremstyle{plain}

\numberwithin{equation}{section}
\newtheorem{theorem}{Theorem}[section]
\newtheorem{proposition}[theorem]{Proposition}
\newtheorem{lemma}[theorem]{Lemma}
\newtheorem{definition}[theorem]{Definition}
\newtheorem{example}[theorem]{Example}
\newtheorem{remark}[theorem]{Remark}
\newtheorem{counterex}[theorem]{Counterexample}

\usepackage{color}
\definecolor{darkred}{rgb}{0.8,0,0}
\definecolor{darkblue}{rgb}{0,0,0.7}
\definecolor{darkgreen}{rgb}{0,0.4,0}

\newcommand{\eps}{\varepsilon}

\newcommand{\R}{{\mathbb R}}

\newcommand{\un}{{\rm 1\kern -2.5pt l}}
\newcommand{\tr}{{\rm Tr}}

\def\u{\mathbf{u}}
\def\uu{\mathbf{uu}}
\def\vv{\mathbf{v}}

\def\xx{\mathbf{x}}

\def\uu{\mathbf{u}}
\def\vv{\mathbf{v}}
\def\zz{\mathbf{z}}
\def\ff{\mathbf{f}_h}

\def\eps{\varepsilon}
\def\R{{\mathbb R}}
\def\N{{\mathcal N}}

\def\H{{\mathcal H}}

\def\K{{\mathcal K}}
\def\A{{\mathcal A}}

\def\eps{\varepsilon}
\def\R{{\mathbb R}}
\def\N{{\mathbb N}}

\def\e{{\mathcal F}}

\def\H{{\mathcal H}}

\def\K{{\mathcal K}}
\def\A{{\mathcal A}}

\def\F{{\mathcal F}}

\def\E{{\mathbb{E}}}

\def\argmin{\mathop{{\rm argmin}}\nolimits}
\def\sign{\mathop{{\rm sign}}\nolimits}
\def\spt{\mathop{{\rm spt}}\nolimits}
\def\Tr{\mathop{{\rm Tr}}\nolimits}

\def\dv{\mathop{{\rm div}}\nolimits}

\def\u{\mathbf{u}}
\def\v{\mathbf{v}}
\def\e{\mathbf{E}}
\def\z{\mathbf{z}}
\def\v{{\bf v}}

\def\e{{\bf e}}
\def\x{{\bf x}}

\def\ttau{\boldsymbol{\tau}}

  \title[]{ Variational Problems for F\"oppl-von K\'arm\'an plates}
 \author[]{Francesco Maddalena, Danilo Percivale, Franco Tomarelli }

 \address{Politecnico di Bari, Dipartimento di Meccanica, Matematica, Management, via Re David 200, 70125 Bari, Italy}
 \email{francesco.maddalena@poliba.it}
 \address{Universit\`{a} di Genova, Dipartimento di Ingegneria Meccanica,
  Piazzale Kennedy, Fiera del Mare, Padiglione D, 16129 Genova, Italy}
  \email{percivale@diptem.unige.it}
\address{Politecnico di Milano, Dipartimento di Matematica,  Piazza Leonardo da Vinci 32, 20133 Milano, Italy}
\email{franco.tomarelli@polimi.it}
\date{\today}  \subjclass{}
\begin{document}
 \maketitle
\begin{abstract}
Some variational problems for a F\"oppl-von K\'arm\'an plate subject to general equilibrated loads are studied.
The existence of global minimizers  is proved under the assumption that
the out-of-plane displacement fulfils homogeneous Dirichlet condition on the whole boundary while the in-plane displacement fulfils nonhomogeneous Neumann condition.\\
If the Dirichlet condition is prescribed only on a subset of the boundary, then the energy may be unbounded from below over the set of admissible configurations, as shown by several explicit conterexamples:
in these cases the analysis of critical points is addressed through  an asymptotic development of the energy functional in a neighborhood of the flat configuration.
By a $\Gamma$-convergence approach we show that critical points of the F\"oppl-von K\'arm\'an energy can be strongly approximated by uniform Palais-Smale sequences of suitable functionals: this property leads to identify relevant features for critical points of approximating functionals, e.g. buckled configurations of the plate. \\
Eventually we perform further analysis as the plate thickness tends to $0$, by assuming that the plate
is prestressed and the energy functional depends only on the transverse displacement around the given prestressed state: by this approach, first we identify
suitable exponents of plate thickness for load scaling,
then we show explicit asymptotic oscillating minimizers as a mechanism to relax compressive states in an annular plate.

\end{abstract}

\tableofcontents
\begin{flushleft}
  {\bf AMS Classification Numbers (2010):\,} 49J45, 74K30, 74K35, 74R10.
\end{flushleft}


\section*{Introduction}
\noindent

The F\"oppl-von K\'arm\'an model   is  widely used as an effective theoretical tool in the study of the mechanical behavior of thin elastic plates, for its ability to describe the interplay between membrane
 and bending effects (see \cite{AP}).
This interplay constitutes the source of a rich phenomenology affecting  not only the macroscopic behavior but also the
occurrence  of local micro-instabilities which are crucial also in the behavior of soft solids, biological tissues, gels (\cite{LNMG}). 
A relevant problem consists in detecting a precise geometric description of such creased equilibrium configurations in dependance of the geometric and constitutive properties of the plate. \\
Despite its long and controversial  history, a rigorous analysis of the well posedness for variational problems associated to the F\"oppl-von K\'arm\'an
functional under general boundary conditions is still far from complete. In particular, the minimization problem under general load conditions is quite subtle. The rigorous derivation of the  F\"oppl-von K\'arm\'an plate  model from three-dimensional
nonlinear elasticity was proved by Friesecke, James and M\"uller in the seminal paper  \cite{FJM} under the assumption of normal forces, while in \cite{LM} the authors carefully analyze the validity of such a theory under in-plane compressive forces and study in detail
the instability issue under suitable coercivity hypotheses (\cite[Theorem 4]{LM}). \\
In this paper we study the existence of minimizers for the
F\"oppl-von K\'arm\'an energy, under general load conditions. In particular,
we deal with Dirichlet and Neumann conditions
for the out-of-plane displacement on the whole boundary while
the in-plane displacement fulfils nonhomogeneous Neumann condition, corresponding to general assumptions on the forces
acting on the plate. The existence of minimizers is proved in several cases by exploiting the techniques introduced in \cite{BBGT},\cite{BT} to circumvent the lack of coerciveness appearing in related nonconvex minimization problems and by taking advantage of some properties of the Monge-Amp\`ere equation (see \cite{RT}, \cite{G}).\\
We exhibit also examples where the energy of admissible configurations is not bounded from below, so that existence of minimizers fails and we turn our attention to the critical points by performing singular perturbation analysis of the functional in a neighborhood of a flat configuration.
This analysis leads to detect  critical points of
the  F\"oppl-von K\'arm\'an energy by suitable approximations of Palais-Smale sequences associated to approximating functionals.
Our procedure allows to single out global buckling configurations, in cases when the plate has
a rectangular shape.
As it is well known, wrinkling type phenomena and other micro instabilities (see \cite{CM},\cite{DPH},\cite{DSV},\cite{PDF},\cite{GO}) manifest themselves in sheets with very small  thickness, therefore we focus our analysis on the behavior as thickness tends to $0$
and highlight the energetic competition of oscillating configurations versus flat equilibrium  configurations.\\
The detailed outline of the paper is as follows.\\
In  Section~\ref{section existence} we prove existence of minimizers for the F\"oppl-von K\'arm\'an energy \eqref{FvK} corresponding to a plate of prescribed thickness $h>0$ under the action of balanced loads in three relevant cases:\\
i) the plate is free at the boundary of a generic Lipschitz open set,
while in plane uniform normal traction or mild uniform normal compression is prescribed on the whole boundary (Theorems~\ref{thm global traction},\,\ref{thm global traction0});\par\noindent
ii) the plate is simply supported on the whole boundary of a strictly convex set (Theorem~\ref{existence1});\par\noindent
iii) the plate is clamped on the whole boundary of a generic Lipschitz open set (Theorem~\ref{existence3}).\par\noindent
Moreover we focus the analysis on the cases when these conditions at the boundary are loosened,
by showing explicit counterexamples where the energy is not bounded from below and minimizers do not exist, even for balanced loads  and fixed thickness $h>0$.
\vskip0.1cm\noindent
Section~\ref{section asymptotic} is devoted to study asymptotic behavior of the energy near a flat configuration; this is achieved by scaling the out-of-plane displacements: Theorem~\ref{ps} shows that critical points of the
F\"oppl-von K\'arm\'an energy, say weak solutions of the corresponding Euler-Lagrange equations,  can be approximately reconstructed by means of \textsl{uniform Palais-Smale sequences} (Definition \ref{ups}) associated to
Gamma-converging simpler functionals (concerning Gamma-convergence and critical points we refer also to \cite{JS}).
This analysis clarifies as some relevant features of critical points, like buckled configurations related to approximating energies, can be recovered by the knowledge of equilibrium configurations related to the flat limit problem (Examples \ref{example 2.7},\,\ref{example 2.8}).\\
In Section 3
we study the limit as $h\!\rightarrow \!0$ of scaled F\"oppl-von K\'arm\'an energy $\F_h$ when in-plane forces in \eqref{FvK} scale as ${\bf f}_{h}\!=\!h^\alpha {\bf f}$:
we show in Theorem~\ref{scaling1} and Counterexample~\ref{infmenoinfinito} that the natural scaling of the problem (entailing convergence of energies and minimizers) occurs if $\alpha\geq 2$:
 under this restriction, if $(\u_{h},w_{h})$ is a minimizer of $\F_h$ then the scaled pairs $(h^{-\alpha}\u_{h},h^{-\alpha/2}w_{h})$ provide a weakly compact sequence in $H^{1}\times H^{2}$ and the corresponding scaled energy converges to a limit energy (Theorem \ref{scaling1} and formula \eqref{gammascaling} therein);
on the other hand,
if $\alpha\in[0,2)$ then the scaled energies may be  unbounded from below  as $h\!\rightarrow\! 0$ even for free plates or simply supported or clamped ones (Counterexample~\ref{infmenoinfinito} and Remark~\ref{unbdclamp}).
\\
The results obtained in Sections 1-3 lead us to examine also the case $\alpha\in[0,2)$, by studying the equilibrium configurations of the plate as
$h\rightarrow 0$ 
through relaxation arguments applied to an energetic functional which takes into account a prestressed state of the plate.
Precisely, in Section~\ref{OscillSect}:
we perform the analysis of corresponding asymptotic minimizers, show
a competition between oscillating and flat equilibria and highlight how this competition is ruled by the mechanical and geometrical parameters: 
oscillating equilibria  act as a mechanism to release compression states in the limit.\\
Eventually we exhibit a list of creased and non creased equilibrium configurations of an annular plate (Examples \ref{rad} -\ref{extang}), together with a general strategy (Remark~\ref{strategy}) to build these examples:
if both eigenvalues in the stress tensor of the prestressed state are strictly positive almost everywhere, then we can expect only the flat minimizer; whereas possible occurrence of oscillating configurations requires the presence of a compressive state on a region of positive measure (Proposition \ref{flat},\,Remark \ref{rmk4.4}).\\
The issues involved in the present article are closely related with a large class of instabilities,
according to recent studies (\cite{BEK}, \cite{BK}, \cite{BK1}, \cite{BCM}, \cite{BKN},\cite{BCDM}, \cite{CM}, \cite{LMP}, \cite{LP}, \cite{LOP}, \cite{PDF}).
\vspace{0.1cm}\\
%
\textit{Notation}.
${\rm Sym}_{2,2}(\R)$ denotes $2\times 2$ real symmetric matrices;
$\mathbf{a}\otimes\mathbf{b}$ denotes the matrix with entries $a_ib_j$,
$\mathbf{a}\odot\mathbf{b}=\frac 1 2(\mathbf{a}\otimes\mathbf{b}+\mathbf{a}\otimes\mathbf{b})$
and $\vert \mathbf a\vert^2=\sum_i a_i^2$ for every $\mathbf{a},\mathbf{b}\!\in\!\R^n$;
moreover $|\mathbb{A}|^2\!=\!\sum_{i,j}A_{ij}^2$ and
$\mathbb{A}\!:\!\mathbb{B}\!=\!\sum_{i,j}\!{A}_{ij}{{B}}_{ij}$,
for every $\mathbb{A},\mathbb{B}\!\in\! {\rm Sym}_{2,2}(\R)$  with entries  respectively $A_{ij},\, B_{ij}$.\\
$H^k(\Omega)$ denotes the Sobolev space of functions in the open set $\Omega\subset\R^2$
whose distributional derivatives up to the order $k$ belong to $L^2(\Omega)$; $H^k_0(\Omega)$ denotes the completion of compactly supported functions in the Sobolev $H^k$ norm; $H^1(\Omega,\R^2)$ denotes the vector fields with components in $H^1(\Omega)$.\\
$-\!\!\!\!\!\int_A v \,d\xx = |A|^{-1}\int_A v \,d\xx$ $\forall $ measurable set $A$ and every integrable function $v$ defined on $A$.
${\bf 1}_{A}(\xx)=1\,$ if $\xx\in A$, $ \ {\bf 1}_{A}(\xx)=0\,$ if $\xx\not\in A$. \
${\chi}_{U}(v)=0\,$ if $v\in U$, ${\chi}_{U}(v)=+\infty\,$ if $v\not\in U$.

\section{Minimization  of F\"oppl-von K\'arm\'an functional}
\label{section existence}
Let $\Omega\subset \R^2$ be a bounded open connected set with Lipschitz boundary
$\partial \Omega$, $\xx=(x_1,x_2)$ denotes the coordinates of points in $\Omega$
referring to the canonical reference frame in $\R^2$ and $s>0$ is the thickness of a thin plate-like region whose  reference configuration is $\Omega\times (-\frac{s}{2},\frac{s}{2})$\,; moreover set $s:=hs_{0}$ where $h$ is an non-dimensional scale factor which remains fixed throughout this Section.\\

Let   $\u:\Omega\rightarrow \R^2$ and  $w:\Omega\rightarrow \R$ be respectively the in-plane
and out-of-plane displacements.
In the geometrical  linear setting the \textsl{stretching tensor} ${\mathbb{D}}$ is given by
\begin{equation}
\label{stretch}
{\mathbb{D}}(\u,w)\,=\,\E(\u)\,+\,\frac{1}{2}\,D w\otimes D w,
\end{equation}
where
\begin{equation}
\label{E}
\E(\u)\,=\,\frac{1}{2}(D\u+D\u^{T})
\end{equation}
denotes the linearized strain tensor. \\
The kernel of $\E$, that is the set of infinitesimal rigid displacements in $\Omega$, is denoted by
\begin{equation}
\label{rigid}
\mathcal{R}:=\{\u: \E(\u)=\mathbf 0\}\
\end{equation}
and $\mathcal R(\u)$ denotes the projection of $\u\in H^{1}(\Omega, \mathbb R^{2})$ on $\mathcal R$.\vskip0.1cm
The elastic energy  of a plate of thickness $hs_0>0$ is the sum of a membrane energy
\begin{equation}
F_{h}^m(\u,w)\ =\ hs_{0}\int_\Omega\,J({\mathbb{D}}(\u,w))\,d\xx
\end{equation}
and a bending energy
\begin{equation}
F_{h}^b(w)\ =\ \frac{h^3 s_{0}^{3}}{12}\int_\Omega J(D^2 w)\,d\xx\,.
\end{equation}
We assume  that for every
${\mathbb{A}}\in  \hbox{\rm Sym}_{2,2}(\R)$  the energy density $J$ is given by
\begin{equation}
\label{density}
J({\mathbb{A}})\ =\frac{E}{2(1-\nu^2)}\left(\vert\Tr({\mathbb{A}})\vert^2-2(1-\nu){\rm det }\,{\mathbb{A}}\right)
\ =\
\frac{E}{2(1+\nu)}\,|\mathbb{A}|^2 +
\frac{E\,\nu}{2(1-\nu^2)}\,|\tr\mathbb{A}|^2 \
\end{equation}
where $E>0$ is the Young modulus and $\nu$ is the Poisson ratio, $-1<\nu<1/2 $.\\
A straightforward consequence of \eqref{density} which will be exploited in subsequent computations is
\begin{equation}\label{coerc}
 c_\nu\,\frac{E}{2}\,|\mathbb A|^{2}\ \le\ J({\mathbb{A}}) \ \le\ C_\nu\,\frac{E}{2}\,|\mathbb A|^{2}
\end{equation}
where $\,0\ <\ c_\nu\ :=\ \min\{(1-\nu)^{-1}, (1+\nu)^{-1}\}\ \leq\  C_\nu\ :=\ \max\{(1-\nu)^{-1}, (1+\nu)^{-1}\}\ < \ +\infty$\,. \
By denoting the unit outer normal to $\partial\Omega$ by $\bf n$, we define
\begin{equation}
\label{A}\begin{array}{ll}
& \A^{0}:=\{w\in H^2(\Omega)\;\vert\; w=\frac{\partial w}{\partial{\bf n}}=0 \hbox{ on }\Gamma\,\} 
\vspace{0.1cm}
\\
&\A^{1}:=\{w\in H^2(\Omega)\;\vert\; w=0 \hbox{ on }\Gamma\}
\vspace{0.1cm}
\\
&\A^{2}:
=H^{2}(\Omega)\\
\end{array}
\end{equation}
where the spaces $\A^{0}=\A^{0}(\Gamma)$, $\A^{1}=\A^{1}(\Gamma)$ actually depend on $\Gamma$. 
We assume in general that
\begin{equation}\label{Gamma}
\Gamma\subset \partial\Omega \quad  \hbox{ is a Borel set s.t.}
\quad\H^{1}(\Gamma)>0 \ .
\end{equation}
Let  \begin{equation}
\label{load}
\mathbf{f}_h\in \ L^2(\partial\Omega,\R^2) \,, \:\:\:  g_{h}\in L^2(\Omega)\,
\end{equation}
respectively be  the densities of a given in-plane load distribution and of a given out-of plane load distribution.


By taking into account the work of external loads and different types of boundary conditions, we define
the \textsl{F\"oppl-von K\'arm\'an  functional}, shortly denoted by 
\textbf{FvK}
in the sequel,
\begin{equation}
\label{FvK}
\begin{array}{ll}
\displaystyle & \F_h (\u,w) = \\
& = \displaystyle
hs_{0}\int_\Omega\! J({\mathbb{D}}(\u,w))\,d\xx
+\frac{h^3 s_{0}^{3}}{12}\!\int_\Omega\! J(D^2 w)\,d\xx
-hs_{0} 
\!\displaystyle\int_\Omega\! g_{h}\,w\; d\xx-hs_{0}
\!\displaystyle\int_{\partial\Omega}\!\!{\bf f}_h\cdot\u\, d\H^1.
\end{array}
\end{equation}
Throughout the paper we choose units of measurement such that $s_{0}=1$.\\
Equilibrium configurations of the plate under prescribed loads $\bf f_{h}$ and $g_{h}$
are obtained by minimizing the functional \eqref{FvK} over $H^1(\Omega,\R^2)\times \mathcal{A}^{i}$, $i=0,1,2$,
corresponding respectively to clamped, simply supported and free plate.
The present Section focuses on issues related to  existence and non existence of these minimizers: we study in detail existence of such minimizers
according to the various choices $i=0,1,2$ of boundary conditions and loads
and we exhibit some counterexamples in which the functional is unbounded from below, hence global minimizers do not exist.
\\ The main obstruction in applying the direct methods of the calculus of variations to
this problem relies in the possible lack of coerciveness of the  functional \eqref{FvK}:
indeed the kernel of the membrane  energy density,  which in general is a subset of the set of solutions of the {\it Monge-Amp\`ere} equation in $\Omega$ (see Lemma~\ref{det} below),
may be too large to allow balancing of the internal membrane energy
versus the effect of external  forces, in order to achieve an equilibrium configuration.
Notwithstanding this difficulty, an existence theorem can be proved either assuming a sign condition on boundary forces, or an homogeneous Dirichlet condition on the transverse displacement.
In the first case the work of the external forces is bounded away from zero on the kernel of  the membrane  energy density, thus allowing the global energy to be bounded from below; in the second one a uniqueness result in the theory of {\it Monge-Amp\`ere} equation implies that the kernel of bending energy reduces to the null transverse displacement
(see also \cite{LMP},  \cite{LP}, \cite{LOP}).
These settings together with a tuning of
some  techniques introduced in
\cite{BBGT} and \cite{BT} yield compactness of minimizing sequences, hence existence of minimizers via the direct method.\\
Assuming $\mathbf f_{h}=f_{h}\mathbf n$,
we prove  existence of minimizers for $\mathcal{F}_h$ in $H^{1}(\Omega,\mathbb R^{2})\times H^{2}(\Omega)$, first under the assumption that $f_{h}$ is a nonnegative constant
(Theorem \ref{thm global traction}), second under the assumption that $f_{h}$ is a small negative constant (Theorem \ref{thm global traction0}).
\begin{theorem}\label{thm global traction}(\textbf{uniform boundary traction of a free plate})\\
Assume that $\Omega\subset \mathbb R^{2}$ is a  bounded connected Lipschitz open set and
\begin{equation}\label{nullres}
\int_{\Omega}g_{h}\,d\xx\,=\int_{\Omega}x_1g_{h}\,d\xx\,=\int_{\Omega}x_2g_{h}\,d\xx\,=\,0 \ ,
\end{equation}
\begin{equation}\label{tractionload}
{\bf f}_h=f_{h}{\bf n} \ \ \hbox{on }\partial \Omega\,,
\qquad f_{h}\ge 0 \ \hbox{is a constant.}
\end{equation}
Then, for every fixed $h>0$,
$\F _h$ achieves a minimum over $H^1(\Omega,\R^2)\times H^2(\Omega)$.
%
\end{theorem}
\begin{proof} In order to achieve the proof it will be enough to show
a minimizing sequence
equibounded in $H^1(\Omega,\R^2)\times H^2(\Omega)$, since $\F_{h}$ is
sequentially l.s.c. with respect the weak convergence in such space.
Due to  $\inf_{H^{1}\times H^{2}} \F_{h}\le \F_{h}({\bf 0},0)\le 0$,  if $\F_{h}(\u_{n},w_{n})\to \inf_{H^{1}\times H^{2}} \F_{h}$ we may suppose $\F_{h}(\u_{n},w_{n})\le 1$ so,
by Divergence Theorem, \eqref{tractionload} and  \eqref{coerc} we also get
\begin{equation}\label{minseq}
\displaystyle c_\nu\frac{h^3 \,E}{24}\int_\Omega\vert D^2 w_{n}\vert^2+
c_\nu\,\frac{h \,E }{2}\int_\Omega\,|{\mathbb{D}}(\u_n,w_{n})|^{2}
 \leq hf_{h}\int_\Omega\hbox{\rm div }\u_{n}+h\int_\Omega g_{h}w_{n}+1.
\end{equation}
Set $\lambda_{n}\!:=\!\|\mathbb E(\u_{n})\|_{L^{2}}$ and suppose by contradiction that  $\sup\lambda_{n}\!=\!+\infty $, hence (up to subsequences without relabeling)
$\lambda_{n}\to +\infty$. Let $\zeta_{n}:=\lambda_{n}^{-1/2}w_{n},\ \ \v_{n}:=\lambda_{n}^{-1}\u_{n}$
and $\mathbf{x}_\Omega$ is the center of mass of $\Omega$. Possibly different constants
denoted by $C$ actually depend only on $\Omega$. Then by substituting in \eqref{minseq} and dividing times $\lambda_n$, we get via \eqref{nullres} and Poincar\`e inequality
\begin{equation}\label{minseq2}
\begin{array}{ll}
&\displaystyle c_\nu\frac{h^3 \,E}{24}\int_\Omega\vert D^2 \zeta_{n}\vert^2+
\lambda_{n}c_\nu\,\frac{h \,E }{2}\int_\Omega\,|{\mathbb{D}}(\v_n,\zeta_{n})|^{2}
 \leq \\
 &\\
 &\displaystyle \le hf_{h}\int_\Omega\hbox{\rm div }\v_{n}
 +\lambda_{n}^{-1/2}\,h\int_\Omega g_{h}\zeta_{n}+\lambda_{n}^{-1}=\\
 &\\
&\displaystyle = hf_{h}\int_\Omega\hbox{\rm div }\v_{n}
+\lambda_{n}^{-1/2}\,h\int_\Omega g_{h}\Big(\zeta_{n}-\,-\!\!\!\!\!\!\int_{\Omega}\zeta_{n}
\,-\, (\mathbf{x}-\mathbf{x}_\Omega)\!-\!\!\!\!\!\!\!\int_{\Omega}D\zeta_{n}
\Big)+\lambda_{n}^{-1}\le \\
 &\\
 &\displaystyle \le hf_{h}\int_\Omega\hbox{\rm div }\v_{n}
 +\lambda_{n}^{-1/2}\,h\,\|g_{h}\|_{L^{2}}^{2}+\lambda_{n}^{-1/2}\,C\int_\Omega\vert D^2 \zeta_{n}\vert^2+\lambda_{n}^{-1}.\\
 \end{array}
\end{equation}
The above inequality together with $\|\mathbb E(\v_{n})\|_{L^{2}}=1$ entail
\begin{equation}
\displaystyle c_\nu\frac{h^3 \,E}{24}\int_\Omega\vert D^2 \zeta_{n}\vert^2+
\lambda_{n}c_\nu\,\frac{h \,E }{2}\int_\Omega\,|{\mathbb{D}}(\v_n,\zeta_{n})|^{2}
 \leq  C
\end{equation}
for large $n$.
Exploiting $\|\mathbb E(\v_{n})\|_{L^{2}}=1,$ once more, we get  $D\zeta_{n}$ are then equibounded in $H^{1}(\Omega,\mathbb R^{2})$,
and, up to subsequences, $\zeta_{n}-\,-\!\!\!\!\!\!\int_{\Omega}\zeta_{n}\to \zeta$ weakly in $H^{2}(\Omega),$ $D\zeta_{n}\to D\zeta$ in $L^{4}(\Omega,\mathbb R^{2})$  due to Rellich Theorem and $\v_{n}\to \v$ weakly in $H^{1}(\Omega,\mathbb R^{2})$.\\
By taking into account \eqref{nullres} we get
\begin{equation}\label{rhs}
hf_{h}\!\int_\Omega\!\hbox{\rm div }\v_{n}+\lambda_{n}^{-1/2}h\!\int_\Omega g_{h}\zeta_{n}=hf_{h}\!\int_\Omega\!\hbox{\rm div }\v_{n}
+\lambda_{n}^{-1/2}h\!\int_\Omega\! g_{h}\big(\zeta_{n}\!- -\! \!\!\!\!\!\int_{\Omega}\zeta_{n}\big)\to hf_{h}\!\int_\Omega\!\!\hbox{\rm div }\v\,.
\end{equation}
By sequential lower semicontinuity together with \eqref{rhs}, \eqref{minseq2} we get
\begin{equation}\label{minseq3}\begin{array}{ll}
&\displaystyle c_\nu\frac{h^3 \,E}{24}\int_\Omega\vert D^2 \zeta\vert^2 \leq \liminf c_\nu\frac{h^3 \,E}{24}\int_\Omega\vert D^2 \zeta_{n}\vert^2\le
\vspace{0.1cm}
\\
&\displaystyle\le  \liminf  \left\{hf_{h}\int_\Omega\hbox{\rm div }\v_{n}+\lambda_{n}^{-1/2}\,h\int_\Omega g_{h}
\big(\zeta_{n}\!- -\! \!\!\!\!\!\int_{\Omega}\zeta_{n}\big)
+\lambda_{n}^{-1}\right\}=  hf_{h}\int_\Omega\hbox{\rm div }\v.\\
\end{array}
\end{equation}
Moreover, by taking into account that $\lambda_{n}\to +\infty$,
\begin{equation}\label{minseq4}
\displaystyle
\lambda_{n}c_\nu\,\frac{h \,E }{2}\int_\Omega\,|{\mathbb{D}}(\v_n,\zeta_{n})|^{2}
 \leq hf_{h}\int_\Omega\hbox{\rm div }\v_{n}+\lambda_{n}^{-1}+\lambda_{n}^{-1/2}\,h\int_\Omega g_{h}
 \big(\zeta_{n}\!- -\! \!\!\!\!\!\int_{\Omega}\zeta_{n}\big)
 \le C
\end{equation}
and  by $D\zeta_{n}\to D\zeta$ in $L^{4}(\Omega,\mathbb R^{2})$,
we have also
\begin{equation}
\displaystyle
c_\nu\,\frac{h \,E }{2}\int_\Omega\,|{\mathbb{D}}(\v,\zeta)|^{2}\le \liminf \displaystyle
c_\nu\,\frac{h \,E }{2}\int_\Omega\,|{\mathbb{D}}(\v_n,\zeta_{n})|^{2}\le C \liminf \lambda_{n}^{-1}=0.
\end{equation}
Hence\\ ${\mathbb{D}}(\v_n,\zeta_{n})\to {\mathbb{D}}(\v,\zeta)=0,$  $ \mathbb E(\v_{n})\to \mathbb E(\v)$  both in $L^2(\Omega,{\rm Sym}_{2,2}(\R))$
and $2\,\hbox{\rm div }\v=-|D\zeta|^{2}$.\\
Therefore  by \eqref{minseq3}
\begin{equation}\label{minseq5}
\displaystyle c_\nu\frac{h^3 \,E}{24}\int_\Omega\vert D^2 \zeta\vert^2 +\frac{1}{2} hf_{h}\int_\Omega |D\zeta|^{2}\le 0
\end{equation}
and by taking into account that $-\! \!\!\!\!\!\int_{\Omega}\zeta=0$
we get $\zeta=0$ and $\mathbb E(\v)=0$, a contradiction since $\|\mathbb E(\v_{n})\|_{L^{2}}=1$ and $\mathbb E(\v_{n})\to \mathbb E(\v)$ in $L^2(\Omega,{\rm Sym}_{2,2}(\R))$.
So $\lambda_{n}\le C$ for some $C>0$ and $\u_{n}-\mathcal R (\u_{n})$ are equibounded in $H^{1}(\Omega, \mathbb R^{2})$ by Korn inequality, while equiboundedness of $w_{n}-\,-\! \!\!\!\!\!\int_{\Omega}w_{n}$  in $H^{2}(\Omega)$ follows from \eqref{minseq}.  Existence of minimizers is then straightforward via direct method.
\end{proof}
If $f < 0$  then the analogous of Theorem \ref{thm global traction} for in-plane compression along the whole boundary
cannot be true, as shown by the next particularly telling Counterexample \ref{countexunifcompr}.
Anyway we can deal also with load corresponding to small negative $f$, as shown by
Theorem \ref{thm global traction0} below.
\begin{counterex}\label{countexunifcompr}
\textbf{(uniform boundary compression)}.\par\noindent
{\rm Assume
\begin{equation}
\label{dom}
\Omega=(-2,2)\times(-1,1)\,, \qquad \Gamma=\{-2\}\times[-1,1],\ \ g_{h}\equiv 0
\end{equation}
\begin{equation}\label{compression in-plane load}
{\bf f}_h=f_{h}{\bf n} \ \ \hbox{on }\partial \Omega,\qquad \hbox{where $f_{h}$
 is a given constant s.t. \ $\ f_h<- \frac {{C_\nu}E}{64}h^{2}$}.
\end{equation}
Then $\inf\F_{h}=-\infty\,$ over both
$\,H^1(\Omega,\R^2)\times\mathcal A^{1}\,$ and
$\,H^1(\Omega,\R^2)\times\mathcal A^{2}$.
\vskip0.2cm
Indeed, let
$$\u=-\frac{(2+x_1)^3}{6}\e_1, \:\:\:\:\:\varphi=\frac{(2+x_1)^2}{2},$$
and $\u_{n}:=n\u,\ \ \varphi_{n}:=\sqrt n \varphi$;
then $2\E(\u_{n})=-D\varphi_{n}\otimes D\varphi_{n}$ and by \eqref{coerc}
\begin{eqnarray*}
\displaystyle\F_{h}(\u_{n},\varphi_{n})&\!\!\leq\!\!&
 \frac{h^3 C_\nu \,nE}{24}\int_\Omega\vert D^2\varphi\vert^2\,d\xx-\,nh\,f_{h}\int_{\partial\Omega}{\bf n}\cdot\u\,d\H^1=\\
&&\\
&\!\!=\!\!&\displaystyle\frac{h^3 C_\nu nE}{24}\int_\Omega\vert D^2\varphi\vert^2\,d\xx-\,nh\,f_{h}\int_{\Omega}\dv\u\,d\xx=\\
&&\\
&\!\!=\!\!&\displaystyle\frac{h^3 C_\nu nE}{24}\!\int_\Omega\vert D^2\varphi\vert^2\,d\xx+\frac{nhf_{h}}{2}\!\int_\Omega|D\varphi|^{2}\,d\xx=  \dfrac{nh\,C_\nu}{3}\left( h^{2}E+64\,f_{h}{C_\nu}^{-1}\right)\to -\infty\,.
\end{eqnarray*}
}
\end{counterex}
\noindent
Referring to the bounded connected Lipschitz open set $\Omega\subset \mathbb R^{2}$,
denote by $K(\Omega)$ the best constant such that\vskip-0.5cm
\begin{equation}\label{sob}
\int_{\Omega}\,\left|\vv - -\!\!\!\!\!\!\int_{\Omega}\vv\right|^{2}d\xx\ \le\ K(\Omega)\int_{\Omega}\left |D\vv\right |^{2}\,d\xx
\qquad \forall\,
\vv\in H^{1}(\Omega,\mathbb R^{2})\,.
\end{equation}
\begin{theorem}\label{thm global traction0}
\textbf{(mild uniform boundary compression of a simply supported plate)}.\par\noindent
Assume that $\Omega\subset \mathbb R^{2}$ is a  bounded connected Lipschitz open
set 
and
\begin{equation}\label{tractionloadbis}
{\bf f}_h=f_{h}{\bf n} \ \ \hbox{on }\partial \Omega
\end{equation}
where $f_{h}$ is a given constant such that
\begin{equation}\label{fh}
f_{h}> -\frac{h^{2}c_{\nu}E}{12\,K(\Omega)}.
\end{equation}
Then, for every fixed $h>0$, $\F_h$ achieves a minimum over
$H^1(\Omega,\R^2)\times H^2(\Omega)\!\cap\! H^1_0(\Omega)$\,.
\end{theorem}
\begin{proof}
Here, by setting $\Gamma=\partial \Omega$, we have $\mathcal{A}^1=H^2(\Omega)\cap H^1_0(\Omega)$.
Let $\F_{h}(\u_{n},w_{n})\to \inf_{H^{1}\times\A^{1}}\F_{h}$ and assuming by contradiction that $\|\mathbb E(\u_{n})\|\!\to\! +\infty$. By arguing as in the proof of Theorem\,\ref{thm global traction}
we can build a sequence $(\v_{n},\zeta_{n})\!\to\! (\v,\zeta)$ weakly in $H^{1}(\Omega, \mathbb R^{2})\!\times\! H^{2}(\Omega)$, $\|\mathbb E(\v_{n})\|\!=\!1,$ $ {\mathbb{D}}(\v_n,\zeta_{n})\to {\mathbb{D}}(\v,\zeta)=\mathbb{O},$  $ \mathbb E(\v_{n})\to \mathbb E(\v)$  both in
$L^2(\Omega,{\rm Sym}_{2,2}(\R))$,
$2\,\hbox{\rm div }\v=-|D\zeta|^{2}$ and
\begin{equation}\label{minseq5bis}
\displaystyle c_\nu\frac{h^3 \,E}{24}\int_\Omega\vert D^2 \zeta\vert^2 +\frac{1}{2} hf_{h}\int_\Omega |D\zeta|^{2}\le 0\, ;
\end{equation}
we emphasize that $\zeta_n=0$ at $\partial\Omega$ entails $-\!\!\!\!\!\int_\Omega D\zeta_n=0$, therefore
$|\int_\Omega g_h\zeta_n|\leq C\,\|g_h\|_{L^2}\,\|D^2\zeta_n\|_{L^2}$ for a suitable constant $C=C(\Omega)$; hence \eqref{minseq5bis} can be achieved even without assuming \eqref{nullres}.
\\
Therefore by taking into account that $\int_{\Omega} D\zeta =0$ (due to $\zeta\in H^1_0$), Poincar\`e inequality \eqref{sob} and assumption \eqref{fh} altogether entail
\begin{equation}\label{minseq5tris}
\displaystyle c_\nu\frac{h^3 \,E}{24 K(\Omega)}\int_\Omega |D\zeta|^{2}+\frac{1}{2} hf_{h}\int_\Omega |D\zeta|^{2}\le c_\nu\frac{h^3 \,E}{24}\int_\Omega\vert D^2 \zeta\vert^2 +\frac{1}{2} hf_{h}\int_\Omega |D\zeta|^{2}\le 0\, ,
\end{equation}
So $D\zeta=0$ and, by $\mathbb{D}(\vv,\zeta)=\mathbb{O}$, $\mathbb E(\v)=\mathbb{O}$, that is a contradiction since $\|\mathbb E(\v_{n})\|_{L^{2}}=1$ and $\mathbb E(\v_{n})\to \mathbb E(\v)$ in $L^2(\Omega,{\rm Sym}_{2,2}(\R))$.
The claim follows by repeating last part of Theorem \ref{thm global traction} proof: here transverse load balancing \eqref{nullres} is not needed, due to boundary condition $\mathcal{A}^1$.
\end{proof}
\begin{remark}
\label{Thm1.3'}
\rm{By inspection of the proof of Theorem \ref{thm global traction0} we deduce also existence theorems
for a plate clamped 
on a possibly proper subset $\Gamma$ of the boundary.
Precisely, assuming $\Omega$ bounded, connected, Lipschitz, \eqref{Gamma}, \eqref{tractionloadbis} with $f_h>-{(h^2c_\nu E)}/{(12\,\widetilde K(\Omega,\Gamma))}$, where $\widetilde K(\Omega,\Gamma)$ is the best constant s.t. $\int_{\Omega}\big |\vv |^{2}\,dx\le K(\Omega,\Gamma))\left\{\int_{\Omega}\left |D\vv\right |^{2}\,dx+\int_{\Gamma}\big |\vv |^{2}\,d\H^{1}\right\}$, then $\mathcal{F}_h$ achieves a minimum over
$H^1(\Omega,\R^2)\times \mathcal{A}^0(\Gamma)$.
\\ Similar claims in\,$H^1(\Omega,\R^2)\times \mathcal{A}^1(\Gamma)$ (for plates supported on\,$\Gamma$) fail,
even by adding assumption $\int_{\Omega}x_1g_{h}d\xx=\!\int_{\Omega}x_2g_{h}d\xx=0$. Indeed, if $\Omega\!=\!(0,1)^2,\,\Gamma\!=\!\{0\}\times[0,1],$
$g_h\equiv 0,\, \mathbf{f}_h=-\lambda^2h^2\mathbf{n}$, then $\inf \mathcal{F}_h=-\infty$, as shown by $\,\uu=-(1/6)(x_1+m)^3\,\mathbf{e}_1\,,\ w_m=\big((x_1+m)^2-m^2\big)/2,\ m\in\mathbb{N}.$}
\end{remark}
Concerning existence of minimizers for $\F_{h}$ in $H^1(\Omega,\R^2)\times\mathcal A^{i}$ for $i=0,1$,
when $\Gamma=\partial \Omega$, that is for clamped and simply supported plates respectively at the whole boundary, in presence of boundary forces which fulfils neither condition \eqref{tractionload} nor conditions \eqref{tractionloadbis}-\eqref{fh} we need to state first  the following Lemma (see also \cite[Proposition 9]{FJM}) which clarifies the link between $\hbox{ker}\, \mathbb D$ and the solutions of the {\it Monge-Amp\`ere} equation in $\Omega$.
\begin{lemma}
\label{det} Let $\Omega \subset \mathbb R^{2}$ be  an open set and assume that  $\u\in H^{1}(\Omega,\mathbb R^{2}), \varphi\in H^{2}(\Omega)$ satisfy
\vskip0.1cm
\centerline{$2\mathbb E(\u)+D\varphi\otimes D\varphi\ =\ 0\ $ in $\ \Omega$\,.}
\vskip0.1cm
Then
$\hbox {\rm det}D^{2}\varphi\equiv 0$ in $\Omega$, where ${\rm det }D^{2}\varphi$ is  the pointwise hessian of $\varphi$.
\end{lemma}
\begin{proof}
Since $\E(\u)$ satisfies the compatibility equation
$$\E_{11,22}+\E_{22,11}=2\E_{12,12}$$
in the sense of ${\mathcal D}'(\Omega),$ we get
\begin{equation*}
\int_{\Omega}\psi_{,2}(\E_{11,2}-\E_{12,1})+\psi_{,1}(\E_{22,1}-\E_{12,2})\,d\xx\ =\ 0\,,
\qquad \forall \psi\in C^{\infty}_{0}(\Omega)\,.
\end{equation*}
Therefore since
$D\varphi\otimes D\varphi = -2\E(\u)$ we get
\begin{eqnarray*}
  \E_{22,1}  &=& -\,\varphi_{,2}\,\varphi_{,12} \\
  \E_{12,2}  &=& -\frac 1 2 \,\varphi_{,2}\,\varphi_{,12}\,-\,\frac 1 2 \,\varphi_{,1}\,\varphi_{,22} \\
  \E_{11,2}  &=& -\,\varphi_{,1}\,\varphi_{,12} \\
  \E_{12,1}  &=& -\frac 1 2 \,\varphi_{,2}\,\varphi_{,11}\,-\,\frac 1 2 \,\varphi_{,1}\,\varphi_{,12}\ .
\end{eqnarray*}
Summarizing
$$\frac 1 2 \int_{\Omega}\psi_{,2}(\varphi_{,11}\varphi_{,2}
-\varphi_{,1}\varphi_{,21})+\psi_{,1}(\varphi_{,1}\varphi_{,22}-\varphi_{,2}\varphi_{,21})\,d\xx=0
\,,
\qquad \forall \psi\in C^{\infty}_{0}(\Omega)\,.$$
that is ${\rm Det }D^{2}\varphi=0$ where $\hbox{Det} D^{2}\varphi$ is the distributional hessian  of $\varphi$. \\
Since $\varphi\in H^{2}(\Omega)$ we have ${\rm det }D^{2}\varphi=\hbox{Det} D^{2}\varphi=0$ in $\Omega$.
\end{proof}
We are now in a position to state and prove an existence theorem for simply supported plates, whose proof relies on a result by Rauch \& Taylor (see  \cite[Theorem 5.1]{RT}) about the Dirichlet problem for  the {\it Monge-Amp\`ere} equation (see also \cite{G}).

\begin{theorem}
\label{existence1}
(\textbf{simply supported plate})\\
If $\Omega\subset\R^2$ is bounded strictly convex and
$\mathbf{f}_h$ is an equilibrated in-plane load distribution, say
\begin{equation}
\label{equil}
\int_{\partial\Omega}\mathbf{f}_h\cdot \zz\ d\H^1=0 \qquad \forall \mathbf{z}\in \mathcal{R}\,.
\end{equation}
Then, for every fixed $h>0$, the {\bf FvK} functional $\F_h
$ in \eqref{FvK} achieves a minimum over $H^1(\Omega,\R^2)\times H^2(\Omega)\!\cap\! H^1_0(\Omega)$.
\end{theorem}
\begin{proof}
Here $\Gamma\!\equiv\!\partial\Omega$ so, referring to \eqref{A}, we look for minimizers of $\mathcal{F}_h$\,over\,$H^1(\Omega,\R^2)\!\times\!\A^{1}=$\\ $H^1(\Omega,\R^2)\times H^2(\Omega)\cap H^1_0(\Omega)$.
The proof will be achieved by showing the existence of a minimizing sequence equibounded in $H^1(\Omega,\R^2)\times H^2(\Omega)$, since $\F_{h}$ is
sequentially l.s.c. with respect the weak convergence in this space.
Due to $\inf_{H^{1}\times\A^{1}} \F_{h}\le \F_{h}(\mathbf{0},0)\le 0$, hence if $\, \F_{h}(\u_{n},w_{n})\to \inf_{H^{1}\times\A^{1}} \F_{h}$ we may suppose $\, \F_{h}(\u_{n},w_{n})\le 1$. So by taking into account
 \eqref{equil} and  \eqref{coerc}   we get via Korn and Poincar\`e inequality
\begin{equation}\label{minseqsimpl}\begin{array}{ll}
&\displaystyle c_\nu\frac{h^3 \,E}{24}\int_\Omega\vert D^2 w_{n}\vert^2+
c_\nu\,\frac{h \,E }{2}\int_\Omega\,|{\mathbb{D}}(\u_n,w_{n})|^{2}\
 \leq \
 h\!\int_\Omega {\bf f}_{h}\cdot \u_{n}+h\int_\Omega g_{h}w_{n}+1\ =
 \vspace{0.1cm}\\
 &\displaystyle =
 h\!\int_\Omega \! {\bf f}_{h}\!\cdot\! \big(\u_{n} \!-\! \mathcal{R}(\u_n)\big) +h\!\int_\Omega \! \!g_{h}w_{n}+1
 \leq
   \|\mathbb E(\u_{n})\|_{\!L^{2}}\| {\bf f}_{h}\|_{\!L^{2}} \! + \! h\,\|g_{h}\|_{\!L^{2}}\| Dw_{n}\!\|_{\!L^{2}}\!+\!1.\\
 \end{array}
\end{equation}
Set $\lambda_{n}:=\|\mathbb E(\u_{n})\|_{L^{2}}$,
assume by contradiction $\lambda_n\to +\infty$ and set $\v_{n}:=\lambda_{n}^{-1}\u_{n}\, \ \zeta_{n}:=\lambda_{n}^{-1/2}w_{n}\,$. By substituting in \eqref{minseqsimpl}
and dividing times $\lambda_n$, via Poincar\`e inequality in $H^2\cap H^1_0$,  we get
\begin{equation}\label{minseqsimpl2}\begin{array}{ll}
&\displaystyle c_\nu\frac{h^3 \,E}{24}\int_\Omega\vert D^2 \zeta_{n}\vert^2+
\lambda_{n}c_\nu\,\frac{h \,E }{2}\int_\Omega\,|{\mathbb{D}}(\v_n,\zeta_{n})|^{2}
 \leq\vspace{0.1cm}\\
 &\displaystyle\le \| {\bf f}_{h}\|_{L^{2}}+\lambda_{n}^{-1/2}\,h\|g_{h}\|_{L^{2}}\| D\zeta_{n}\|_{L^{2}}+\lambda_{n}^{-1}\le\vspace{0.1cm}\\
 &\displaystyle\le C+\lambda_{n}^{-1/2}\,h\int_{\Omega}|D\zeta_{n}|^{2}\le C+\lambda_{n}^{-1/2}\,h\int_{\Omega}|D^{2}\zeta_{n}|^{2}
 \end{array}
\end{equation}
thus obtaining as in  the proof of Theorem \ref{thm global traction}
\begin{equation}\label{minseqsimpl3}
\displaystyle c_\nu\frac{h^3 \,E}{24}\int_\Omega\vert D^2 \zeta_{n}\vert^2+
\lambda_{n}c_\nu\,\frac{h \,E }{2}\int_\Omega\,|{\mathbb{D}}(\v_n,\zeta_{n})|^{2}
 \leq  C'
\end{equation}
for a suitable $C'>0$.
Since $\|\mathbb E(\v_{n})\|_{L^{2}}=1,$ $D\zeta_{n}$ are then equibounded in $H^{1}(\Omega,\mathbb R^{2})$
so, up to subsequences, $\zeta_{n}\to \zeta$ weakly in $H^{2}(\Omega),$ $D\zeta_{n}\to D\zeta$
strongly in $L^{4}(\Omega,\mathbb R^{2})$,  $\v_{n}\to \v$ weakly in $H^{1}(\Omega,\mathbb R^{2})$
and ${\mathbb{D}}(\v_{n},\zeta_n)\rightarrow 0$ strongly in $L^2(\Omega)$.
Hence
\begin{equation}\label{det0}
2\E(\v_n)+D \zeta_{n}\otimes D \zeta_{n}\rightarrow 2\E(\v)+D \zeta\otimes D \zeta={\mathbb O}\qquad\hbox{\rm strongly in }L^2(\Omega,{\rm Sym}_{2,2}(\R))
\end{equation}
and $\E(\v_n)\to  \E(\v)$ strongly in $L^2(\Omega,{\rm Sym}_{2,2}(\R))$.
Then by Lemma \ref{det} we have
${\rm det }D^{2}\zeta=0$  and by taking into account that $\Omega$ is strictly convex and $\zeta=0$ on the whole $\partial\Omega$
by uniqueness Theorem 5.1 in \cite{RT} we get  $\zeta\equiv 0$ in  $\Omega$. This implies $\E(\v)=  -\frac{1}{2}D\zeta\otimes D\zeta = \mathbb{O}$, which is a contradiction since $\|\mathbb E(\v_{n})\|_{L^{2}}=1$.
Hence $\lambda_{n}\le C$ for suitable $C>0$,  so $\u_{n}-\mathcal R (\u_{n})$ are equibounded in $H^{1}(\Omega, \mathbb R^{2})$ and equiboundedness of $w_{n}$ in $H^{2}(\Omega)$ follows from \eqref{minseqsimpl3}. Existence of minimizers is obtained via direct method.
 \end{proof}


Existence of minimizers may fail when $\Gamma\not\equiv\partial\Omega$
even if the in-plane load $\mathbf{f}_h$ is equilibrated, as shown by the next Counterexample.
\begin{counterex}\label{sheartang}\!\!\!(\textbf{buckling under in-plane shear}) 
{\rm Fix $\gamma\!>\!0,\, 
\eps\!>\!0,\, h^2\!<\!\gamma/(6EC_\nu) $ and
$$\Omega_\eps= \{(x_{1},x_{2}): |x_{1}|< 2+\eps(1-x_{2}^{2}),\  |x_{2}|< 1+\eps(4-x_{1}^{2})\}\,,$$
\begin{equation}\label{Sigma1.4}
\Gamma_\eps\ = \partial\Omega_\eps\cap \{(x_{1},x_{2}): |x_{1}-x_{2}|\ge 1\}\,,\end{equation}
\begin{equation}\label{f1}
{\bf f}_h:=\gamma\ttau\left({\bf 1}_{\Sigma^{2,\pm}}-{\bf 1}_{\Sigma^{1,\pm}}\right)\,,
\end{equation}
%
where $\ttau$ denotes the counterclockwise oriented  unit vector  tangent  to $\partial\Omega_\eps= \Sigma^{1,\pm}_\eps\cup\Sigma^{2,\pm}_\eps$ and
$$\Sigma^{1,\pm}_\eps=\{(x_{1},x_{2}):  |x_{1}|\le 2,\ x_{2}=\pm( 1+\eps(4-x_{1}^{2}))\}$$
$$\Sigma^{2,\pm}_\eps=\{(x_{1},x_{2}):  |x_{2}|\le 1,\ x_{1}=\pm( 2+\eps(1-x_{2}^{2}))\}.$$
We claim that there exists $\widetilde\eps$ such that $\,\inf\mathcal{F}_h=-\infty\,$ over $H^{1}(\Omega_{\widetilde\eps},\R^2)\times\A^{1}$
under the assumptions listed above, notwithstanding the strict convexity of $\Omega_{\widetilde\eps}$ and the fact that condition \eqref{equil} holds true.
\\
Indeed, let $\psi\in C^{1,1}(\R)$ be an even function, with $\spt \psi\subset [-1,1]$, $\psi'=-1$ in $[1/4,3/4]$ and $|\psi''|\leq 4$ in $\R$.
We set   $\varphi(x_1,x_2)=\psi(x_1-x_2)$ and define $w_{n}:=\sqrt n\varphi$ and $\u_{n}:=n\u$\,, where
$$u_2(x_1,x_2)\,=\,-u_1(x_1,x_2)\,=\ \frac{1}{2}\int_{-1}^{x_1-x_2}\vert\psi'(\tau)\vert^2\,d\tau\ .$$
By setting $\Omega_{0}:=(-2,2)\times (-1,1)\subset \Omega_\eps$,
there is $C>0$ such that for every $0<\eps\leq 1$
 \begin{equation*}\label{scarbord1}
 \left|\,
 \int_{\partial\Omega_\eps}{\bf f}_h\cdot\u\ d\mathcal{H}^{1}\ -\ \int_{\partial \Omega_{0}}{\bf f}_h\cdot\u\ d\mathcal{H}^{1}
 \, \right| \ \leq\ C\,\eps\ ,
 \end{equation*}
 hence by \eqref{f1} and there exists $\widetilde\varepsilon\in(0,1)$ such that
 \begin{equation}\label{scarbord2}
 \int_{\partial\Omega_{\widetilde\varepsilon}}{\bf f}_h\cdot\u\ d\mathcal H^{1}\ \geq\ \int_{\partial\Omega_{0}}{\bf f}_h\cdot\u\ d\mathcal H^{1}-\dfrac{\gamma}{2}\ \,= \ \,\gamma\int_{\Omega_{0}}2\,\E_{12}(\u)\ d\xx\,-\,\dfrac{\gamma}{2}\ .
 \end{equation}
So\vskip-0.4cm
$$u_{1,1}(x_1,x_2)=-\frac{1}{2}\vert\psi'(x_1-x_2)\vert^2=-\frac{1}{2}\varphi_{,1}^2,$$
$$u_{2,2}(x_1,x_2)=-\frac{1}{2}\vert\psi'(x_1-x_2)\vert^2=-\frac{1}{2}\varphi_{,2}^2,$$
$$\frac{u_{1,2}+u_{2,1}}{2}=\frac{1}{2}\left[ \frac{1}{2}  \vert\psi'(x_1-x_2)\vert^2+
\frac{1}{2}  \vert\psi'(x_1-x_2)\vert^2  \right]=\frac{1}{2}\vert\psi'(x_1-x_2)\vert^2=-\frac{1}{2}\varphi_{,1}\varphi_{,2}$$
that is  $\E(\u_{n})=-\frac{1}{2}Dw_{n}\otimes Dw_{n}$ and
moreover, by \eqref{coerc}, \eqref{scarbord2} and $ \varphi,_2=-\varphi,_1 $ we deduce
\begin{equation}\label{ineqinfty}
\begin{aligned}
&\F_{h}(\u_{n},w_{n}) \leq\
C_\nu\,\frac{h^3 \,nE}{24}\int_{\Omega_{0}}\vert D^2\varphi \vert^2\,dx
\,+\, C_\nu\,\frac{h^3 \,nE}{24}\int_{\Omega_{\widetilde\varepsilon}\setminus\Omega_{0}}\vert D^2\varphi \vert^2 d\xx+
\vspace{0.1cm}\\
&\qquad\qquad\qquad +hn\gamma\int_{\Omega_{0}}
\varphi_{,1}\varphi_{,2}d\xx+h\dfrac{\gamma}{2}\ \le
\vspace{0.25cm}\\
& \leq
 C_\nu\,\frac{8h^3 \,nE}{3}\Big(\big\vert\{(x_1,x_2)\!\in\!\Omega_{0}:\ 4\vert x_1\!-x_2\vert\leq 1 \ \,\hbox{or}\ \, 3\leq 4\vert x_1-x_2\vert\leq 4   \}\big\vert
 + |\Omega_{\widetilde\varepsilon}\setminus\Omega_{0}|\Big)+\vspace{0.15cm}\\
&
\qquad\qquad\qquad -hn\gamma \vert\{(x_1,x_2)\!\in\!\Omega_{0}\!:\! 1 \leq 4\vert x_1-x_2\vert\leq 3   \}\vert+hn\dfrac{\gamma}{2}\ \le
\vspace{0.2cm}\\
 & \le
\ 3C_\nu E h^3n\,-\,hn\dfrac{\gamma}{2} \to -\infty
 \end{aligned}
\end{equation}
\noindent
 as $n\to +\infty$ whenever $6EC_\nu\,h^2<\gamma $ thus proving the claim.\quad $\square$
}
\end{counterex}
Clearly Theorems\,\ref{thm global traction},\,\ref{thm global traction0},\,\ref{existence1} hold  for the clamped plate too: minimization in\,$H^1(\Omega,\R^2)\!\times\!\mathcal A^{0}.$
Even better, in the case of clamped plate we can drop both convexity assumption on $\Omega$
and  equilibrated out-of-plane load \eqref{nullres} as it is shown by the next result.
\begin{theorem}
\label{existence3}
(\textbf{clamped plate})\\
If  $\ \Omega$ is a  bounded connected Lipschitz open set  
and \eqref{equil} holds,
then for every fixed $h>0$ the functional $\F_h
$ in \eqref{FvK}
achieves its minimum over $H^1(\Omega,\R^2)\times H^{2}_{0}(\Omega)$.
\end{theorem}
\begin{proof} Again we need only to exhibit an equibounded minimizing sequence. Indeed,
as in the proof of Theorem \ref{existence1} if $\, \F_{h}(\u_{n},w_{n})\to \inf_{H^{1}\times H^{2}_{0}} \F_{h}$ we may suppose $\, \F_{h}(\u_{n},w_{n})\le 1$.
Then, since $\Gamma\!=\!\partial\Omega$ entails $H^{2}_{0}(\Omega)\!=\!\A^{0}\!\subset \! \A^{1}$, by setting $\lambda_{n}\!:=\!\|\mathbb E(\u_{n})\|_{L^{2}}$,
$\v_{n}\!:=\!\lambda_{n}^{-1}\u_{n}\,, \, \zeta_{n}\!:=\!\lambda_{n}^{-1/2}w_{n}\,$  and assuming $\lambda_{n}\to +\infty$, arguing as in the previous proofs we achieve
the estimates \eqref{minseqsimpl}, \eqref{minseqsimpl2}, \eqref{minseqsimpl3}. Then the sequence $D\zeta_{n}$ is equibounded in $H^{1}(\Omega,\mathbb R^{2})$
so, up to subsequences, $\zeta_{n}\to \zeta$ weakly in $H^{2}(\Omega),$ $D\zeta_{n}\to D\zeta$ in $L^{4}(\Omega,\mathbb R^{2})$,  $\v_{n}\to \v$ weakly in $H^{1}(\Omega,\mathbb R^{2})$
and ${\mathbb{D}}(\v_{n},\zeta_n)\rightarrow \mathbb{O}$ in $L^2(\Omega,{\rm Sym}_{2,2}(\R))$.
Hence
\begin{equation}\label{det00}
2\E(\v_n)+D \zeta_{n}\otimes D \zeta_{n}\rightarrow 2\E(\v)+D \zeta\otimes D \zeta=
\mathbb{O}\qquad\hbox{\rm strongly in }L^2(\Omega,{\rm Sym}_{2,2}(\R)),
\end{equation}
$\E(\v_n)\to  \E(\v)$ strongly in $L^2(\Omega,{\rm Sym}_{2,2}(\R))$ and
 by Lemma \ref{det} we have
 ${\rm det }D^{2}\zeta=0$ in the whole $\Omega$.
Since $\zeta\equiv\frac{\partial\zeta}{\partial\bf n}\equiv 0$ on  $\partial\Omega$, there exists a disk $\widetilde\Omega$ (bounded and strictly convex!)  such that $\Omega\subset\widetilde\Omega$ and the  trivial extension $\widetilde\zeta$ of $\zeta$ in $\widetilde\Omega$ belongs to $H^{2}_{0}(\widetilde\Omega)$. Therefore $\det D^{2}\widetilde\zeta=0$ on $\widetilde\Omega$ and still by Theorem 5.1 in \cite{RT} we get $\widetilde\zeta\equiv 0$ in $\widetilde\Omega$ hence $\zeta\equiv 0$ in $\Omega$. Then by \eqref{det00} $\E(\v)=\mathbb{O}$, a contradiction since $\|\mathbb E(\v_{n})\|_{L^{2}}=1$.
\end{proof}

\section{Critical points nearby a flat configuration}
\label{section asymptotic}
When existence of global minimizers fails because the energy is unbounded from below, it is natural to investigate the
structure of local minimizers or, more in general of critical points.
Since the nonlinearity in the $\mathbf{FvK}$ functional
relies in the interaction between membrane and bending contributions, we will focus in this section on the asymptotic analysis
of  critical points in the neighborhood  of a flat configuration, i.e. we will study the behavior for
small out-of-plane displacements.
Throughout this section we assume that $h\!>\!0$ is fixed and
\begin{equation}\label{(2.1)}
g_{h}\equiv 0\, 
\end{equation}
that is, we restrict our analysis to the case of in-plane load acting on a plate of prescribed thickness.
Assume ${\bf f}_h\!\in \!L^2(\partial\Omega,\R^2)$ and \eqref{equil} holds true.
For every $({\bf u}, w)\!\in\! H^{1}(\Omega,\R^2)\!\times \!H^{2}(\Omega)$, referring to
\eqref{stretch}\,-\,\eqref{FvK},
we enclose boundary conditions in the functional, by setting
\begin{equation}
\label{FvKi}
\displaystyle\F_h^{i}(\u,w)=\left\{\begin{array}{ll}
\displaystyle\F_h (\u,w) & \ {\rm if }\:\: \u\in H^1(\Omega,\R^2),\:w\,\in\A^{i}\,,\\
\\
+\infty
& \ \hbox{\rm  otherwise \,, }\\
\end{array} \right.
\end{equation}
\begin{equation}
\F_{h,\varepsilon}^{i}(\u,w)=\F_{h}^{i}(\u,\varepsilon w)\,,\qquad \forall\,\varepsilon>0\,. \qquad\qquad\qquad\qquad\ \  \phantom{:}
\end{equation}
By noticing that $\F_{h,0}:=\F_{h,0}^{i}$ actually is independent of $i$, we also set
\begin{equation}\label{E i h eps}
\mathcal E_{h,\eps}^{i}(\u,w)=\eps^{-2}\left(\F_{h,\eps}^{i}(\u,w)-
\min_{H^{1}(\Omega,\R^2)}\F_{h,0}\right),
\end{equation}
\begin{equation}
\label{E}
\mathcal E_{h}^{i}(\u,w)=\!\left\{\!\!\begin{array}{ll} \displaystyle F_{h}^{b}(w)+\frac{h}{2}\int_{\Omega}J'(\E(\u)) 
: Dw\otimes Dw\,d\xx& \text{if}\ (\u,w)\in\{\argmin\mathcal{F}_{h,0}\}\times \A^{i}\\
&\\
\displaystyle+\infty\  &   \text{else in} \ \ H^{1}(\Omega,\R^2)\times H^{2}(\Omega)\end{array}\right.
\end{equation}
where
\begin{equation}\label{J'}
J'(\mathbb{A})= \frac {E}{1+\nu}\mathbb{A}+\frac{E\nu}{1-\nu^2}(\Tr \mathbb{A})\mathbb{I}
\end{equation}
denotes the derivative of $J$.
\\
Functionals $\mathcal E_{h,\eps}^{i}$ and $\F_{h,\varepsilon}^{i}$ are linked via the following result
\begin{proposition}\label{Gammalim}
$\ \displaystyle \mathcal E_{h}^i \ \ =\ \ \mathop{\Gamma\,\lim}_{\varepsilon\to 0_+}   \ \mathcal{E}_{h,\eps}^i \ $.\\
Precisely, the following relations hold true:\par\bigskip\noindent
i) for every $(\u_{\eps},w_{\eps})\rightharpoonup (\u,w)$ in $w-H^{1}\times H^{2}$ we have
\begin{equation}\label{ginf}
\liminf_{\eps\to 0}\mathcal E_{h,\eps}^{i}(\u_{\eps},w_{\eps})\ge \mathcal E_{h}^{i}(\u,w);
\end{equation}
ii) for every $(\u,w)\in H^{1}\times H^{2}$ there exists $(\widetilde\u_{\eps},\widetilde w_{\eps})\rightharpoonup (\u,w)$ in $w-H^{1}\times H^{2}$ such that
\begin{equation}\label{gsup}
\lim_{\eps\to 0}\mathcal E_{h,\eps}^{i}(\widetilde\u_{\eps},\widetilde w_{\eps})= \mathcal E_{h}^{i}(\u,w).
\end{equation}
\end{proposition}
\begin{proof} Let $(\u_{\eps},w_{\eps})\rightharpoonup (\u,w)$ in $w-H^{1}\times H^{2}$: by convexity we have
\begin{equation}\begin{array}{ll}
&\F_{h,\varepsilon}^{i}(\u_\varepsilon,w_\varepsilon)\geq \eps^{2}F_{h}^{b}(w)+h \int_\Omega\! J(\E(\u_\varepsilon))\,dx+\\
&\\
&+\dfrac{h\varepsilon^2}{2}
\int_\Omega\!\,J\,'(\E(\u_\varepsilon)) 
: Dw_\varepsilon\otimes Dw_\varepsilon\,dx-h \int_{\partial\Omega}{\bf f}_h\cdot\u_\varepsilon\,d\H^1\ge\\
&\\
&\ge\eps^{2}F_{h}^{b}(w)+\dfrac{h\varepsilon^2}{2}
\int_\Omega\!\,J\,'(\E(\u_\varepsilon)) 
: Dw_\varepsilon\otimes Dw_\varepsilon\,dx+\min \mathcal F_{h,0}
\end{array}
\end{equation}
and by taking into account  that
$D w_\varepsilon\otimes D w_\varepsilon\rightarrow D w\otimes D w$
strongly in $L^2(\Omega,{\rm Sym}_{2,2}(\R))\,$ and
$\, J'(\E(\u_\varepsilon))\rightharpoonup \,J\,'(\E(\u))$ weakly   in $L^2(\Omega,{\rm Sym}_{2,2}(\R))$, we get
 $$\liminf_{\eps\to 0}\mathcal E_{h,\eps
 }^{i}(\u_{\eps},w_{\eps})\ge \mathcal E_{h}^{i}(\u,w)$$
 and i) is proven. The proof of ii) is achieved by taking  $(\widetilde\u_{\eps},
 \widetilde w_{\eps})\equiv (\u,w)$.
\end{proof}
We recall that  if  ${\mathcal I}:X\rightarrow \R$ is any $C^1$ functional defined on a Banach space $X$ then $\overline x\in X$ is a critical point for ${\mathcal I}$ if $\mathcal I'(\overline x)=0$ where $\mathcal I': X\to X^{*}$ denotes the Gateaux differential of $\mathcal I$.\\
Due to formula  \eqref{EL} below, $\F_{h,\varepsilon}^{i}$ is a $C^1$ functional in the Hilbert space $H^{1}(\Omega,\R^2)\times \A^{i}$: precisely, for every $(\u,w)\in H^{1}(\Omega,\R^2)\times \A^{i}$ the Gateaux differential  of  $\F_{h,\varepsilon}^{i}$ at $(\u,w)$ is given by
$$(\F_{h,\varepsilon}^{i})'(\u,w)[(\zz,\omega)]\ =\
\Big(\,\tau_{1}(\u,w)[\zz]\,,\, \tau_{2}(\u,w)[\omega]\,\Big) \ ,\qquad
\ \forall \zz\in H^{1}(\Omega,\R^2) \,,\forall \omega\in \A^{i}\,,$$
where
\begin{equation}
\label{EL}
\begin{array}{ll}
\displaystyle\tau_{1}(\u,w)[\zz]\, :=\ h\,\int_\Omega\!J'\left({\mathbb E}(\u)+\frac{\eps^{2}}{2}D w\otimes D w\right)\! : 
{\mathbb E}(\zz)\,-\,h\int_{\partial\Omega}\!{\mathbf f}_h\cdot\zz \ ,
\\
&\\
\!\! \tau_{2}(\u,w)[\omega]\, :=\ \displaystyle \eps^{2}\,\frac{h^3}{12}\int_\Omega\!J'(D^{2}w) : 
D^{2}\omega\,+\,\eps^{2}\,h
\int_\Omega\!J'\left({\mathbb E}(\u)+\frac{\eps^{2}}{2}D w\otimes Dw\right) \! : 
Dw\odot D\omega \ .
\\
\end{array}
\end{equation}
$\big(\tau_{1}(\u,w)[\zz] , \tau_{2}(\u,w)[\omega]\big)$ is replaced by the shorter notation $\big(\,\tau_{1}[\zz]\,,\, \tau_{2}[\omega]\,\big)$, whenever the dependance on fixed choice for $(\u,w)$ is understood.
Actually \eqref{EL} provides the explicit information that $(\F_{h,\varepsilon}^{i})'(\u,w)$ depends continuously on $(\u, w)$.\\
Hence the F\"oppl-von Karman plate equations in weak form together with boundary conditions can be written as follows:
\begin{equation}\label{vKeqs}
\left\{
\begin{array}{ll}
  \u,\,w \in H^1(\Omega,\R^2)\times \mathcal{A}^i   \,, \\
  \tau_{1}(\u,w)[\zz]         \ = \ 0    &\quad \forall \zz\in H^{1}(\Omega,\R^2) \,,\\
  \tau_{2}(\u,w)[\omega] \ = \ 0   &\quad \forall \omega\in \A^{i} \,.
\end{array}
\right.
\end{equation}
Clearly  $(\mathcal E_{h,\varepsilon}^{i})'(\u,w)=\eps^{-2}(\F_{h,\varepsilon}^{i})'(\u,w)$
hence  $\F_{h,\varepsilon}^{i}$ and $\mathcal E_{h,\varepsilon}^{i}$ have the same critical points. Moreover
if $\u_*\in \argmin \mathcal{F}_{h,0}$ then $\tau_{2}(\u_{*},0)\equiv 0$ and $(\u_{*},0)$ is a critical point for  $\F_{h,\varepsilon}^{i}$.\\
The next definition tunes the standard notion of
\textit{Palais-Smale sequence} to the present context.
\begin{definition}
\label{ups}
Let  ${\mathcal I}_{\eps}:X\!\rightarrow \!\R$ be a sequence of  $C^1$ functionals and $X$ be a Banach space
$X$.  A sequence $\{x_\eps\}\subset X$ is a \textbf{uniform Palais-Smale sequence} if  there exists $C>0$ such that  $\mathcal I_{\eps}(x_\eps)\le C$ and
$\parallel \mathcal I_{\eps}'(x_\eps)\parallel_{X^*}\rightarrow 0,$ as $\eps\rightarrow 0_{+}$.
\end{definition}
Notice that the above definition reduces to the usual notion of Palais-Smale sequences
when ${\mathcal I}_{\eps}\equiv \mathcal I$ for every $\eps>0$.
Let $\u_*\in \argmin \mathcal{F}_{h,0}$, we denote by
$\mathcal{K}^i_h(\u_*)$ the set of critical points in $\A^i$ of $\mathcal{E}^i_h(\u_*,\cdot)$ that is
\begin{equation}\label{critpointsE}
  \mathcal{K}^i_h(\u_*)=\{w\in\A^i : \tau_2(\u_*,w)[\omega]=0\,, \ \forall \,\omega\in \A^i \}\ .
  \end{equation}
Next result shows that  any critical point of $\mathcal{E}_h(\u_*,\cdot)$ in $\A^i$ can be approximated by  a uniform Palais-Smale sequence of $\mathcal{E}^{i}_{h,\varepsilon}$
 whose energy converges to the energy of the critical point itself.
 \begin{theorem}
\label{ps}
Let $\u_*\in \argmin \mathcal{F}_{h,0}$, $w\in \mathcal{K}^i_h(\u_*)$ and
$\z_w\in\argmin\mathcal{Q}_{w}(\z)$, where
\begin{equation}\mathcal{Q}_{w}(\z)\ := \ \int_\Omega\!\,J\Big(\mathbb E(\z)+\frac{1}{2}D{w}\otimes D {w}\Big)\,d\xx
\end{equation}
Then $\{(\u_*+\varepsilon^2\z_w,w)\}_{\eps>0}$ is a uniform Palais-Smale sequence for $\mathcal{E}^{i}_{h,\varepsilon}$ and
\begin{equation*}\lim_{\eps\to 0_+}\mathcal{E}^{i}_{h,\varepsilon}(\u_*+\varepsilon^2\z_w,w) \ = \ \mathcal{E}^i_h(\u_*,w)\ .
\end{equation*}
\end{theorem}
\begin{proof}
We have to prove the following conditions
\begin{itemize}
\item[a)] \ \,$\mathcal{E}^{i}_{h,\varepsilon}(\u_*+\varepsilon^2\z_w,w)\ \,\leq \ \,C\ \, <+\infty\,,\quad \forall\varepsilon\in(0,1]$\,,\\
\item[b)] \ $(\mathcal{E}^{i}_{h,\varepsilon})'(\u_*+\varepsilon^2\z_w,w)\rightarrow 0$ strongly in $(H^1(\Omega,\R^2)\times\A^i)^*$\,,\\
\item[c)] \ $\displaystyle\lim_{\eps\to 0_{+}}\mathcal{E}^{i}_{h,\varepsilon}(\u_*+\varepsilon^2\z_w,w)= \mathcal{E}^i_h(\u_*,w).$\\
\end{itemize}
We first prove c), which implies a) too.
Indeed
\begin{equation*}
\begin{array}{ll}
&\hspace{-0.7cm}\mathcal{E}^{i}_{h,\varepsilon}(\u_*+\varepsilon^2\z_w,w)\ =\
\varepsilon^{-2}\left[ \F_h^1(\u_*+\varepsilon^2\z_w,\eps w)-\F_{0,h}(\u_*)\right]\ = \\
& \\
=&\displaystyle\varepsilon^{-2}\left[\frac{h^3}{12}\int_\Omega\! J(\varepsilon D^2 w)\,dx+h\int_\Omega\!J(\mathbb E(\u_{*})+\varepsilon^2\mathbb E(\z_w)+\frac{\varepsilon^2}{2}D w\otimes D w)\,d\xx\right]-\\
&\\
&\displaystyle-\varepsilon^{-2}\left[h\int_\Omega\!J(\mathbb E(\u_{*}))+\varepsilon^2 h\int_{\partial\Omega}\!{\mathbf f}_h\cdot\z_w \right] \ =\\
&\\
=&\displaystyle\varepsilon^{-2}\left[\frac{h^3}{12}\varepsilon^2\int_\Omega\! J(D^2 w)\,d\xx+h\int_\Omega\!J(\mathbb E(\u_{*})+\varepsilon^4h\int_\Omega\!J\left(\mathbb E(\z_w)+\frac{1}{2}D w\otimes D w\right)\,d\xx\right]+\\
&\\
&\displaystyle +\varepsilon^{-2}\left[\varepsilon^2 h\int_\Omega\!J'(\mathbb E(\u_*)):\left(\mathbb E(\z_w)+\frac{1}{2}D w\otimes D w\right)\,d\xx-h\int_\Omega\!J(\mathbb E(\u_*))-\varepsilon^2\int_{\partial\Omega}\!{\mathbf f}_h\cdot\z_w\right] \ =\\
&\\
=&\displaystyle\frac{h^3}{12}\int_\Omega\! J(D^2 w)\,d\xx+\varepsilon^2h\int_\Omega\!J\left(\mathbb E(\z_w)+\frac{1}{2}D w\otimes Dw\right)\,d\xx+\\
&\\
&\displaystyle + h\int_\Omega\!J'(\mathbb E(\u_*)):\left(\mathbb E(\z_w)+\frac{1}{2}D w\otimes D w\right)\,d\xx-h\int_{\partial\Omega}\!{\mathbf f}_h\cdot\z_w \ =\\
&\\
=&\displaystyle\frac{h^3}{12}\int_\Omega\! J(D^2 w)\,d\xx+\varepsilon^2h\int_\Omega\!J\left(\mathbb E(\z_w)+\frac{1}{2}D w\otimes Dw\right)\,d\xx+\\
&\\
&\displaystyle +\frac{h}{2}\int_\Omega\!J'(\mathbb E(\u_*)): D w\otimes D w\,d\xx\\
\end{array}
\end{equation*}
since, due to minimality of $\u_{*}$,
$$\int_\Omega\!J'(\mathbb E(\u_*)):\mathbb E(\z_w)\,d\xx-\int_{\partial\Omega}\!{\mathbf f}_h\cdot\z_w=0.$$
Hence $\lim_{\eps\to 0}\mathcal{E}^{i}_{h,\varepsilon}(\u_*+\varepsilon^2\z_w,w)= \mathcal{E}^i_h(\u_*,w)$ as claimed.\\
Eventually we prove b). By recalling \eqref{E i h eps} and \eqref{EL}, we get for every $\zz\in H^1(\Omega,\R^2)$ and $\omega\in\A^i$
$$(\mathcal{E}^{i}_{h,\varepsilon})'(\u_*+\varepsilon^2\z_w,w)[(\zz,\omega)]
\ =\ \eps^{-2}\,\Big(\tau_1\big(\,\u_{*}+\varepsilon^2\z_w,w\,\big)[\zz]\,,\,
\tau_2\big(\,\u_{*}+\varepsilon^2\z_w,w\,\big)[\omega]\,\Big)\ .$$
%
Since $\z_w\in\argmin\mathcal{Q}(\z)$\,, $\,\u_{*}\in \argmin \F_{h,0}$ and
$w\in \K^i_w(\u_*)$  we get:
$$ \tau_1(\u_{*},0)[\zz]=0\ \ \forall \zz\in H^1(\Omega,\R^2)\,, \qquad\tau_2(\u_{*},w)[\omega]=0\ \
\forall  \omega \in\A^i\,,$$
$$\eps^{-2}\,\tau_2(\u_*+\eps^2\zz_w,w)[\omega]\,=\,\eps^{2}\int_\Omega\!
J'\left({\mathbb E}(\z_w)+\frac{1}{2}Dw\otimes Dw\right): Dw\otimes D\omega\ .$$
The above relationships together with \eqref{critpointsE} imply
$$\sup_{\parallel(\zz,\omega)\parallel\leq 1}\big\vert\,(\mathcal{E}^{i}_{h,\varepsilon})'(\u_*+\varepsilon^2\z_w,w)[(\zz,\omega)]
\,\big\vert\rightarrow 0\:\:\:\hbox{\rm as }\;\;\varepsilon\rightarrow 0,$$
where $\parallel\!(\zz,\omega)\!\parallel\,=\,\parallel\!\zz\!\parallel_{H^1}+
\parallel\!\omega\!\parallel_{H^2}$\,, \,thus proving $b)$\,.
\end{proof}
\begin{remark}{\rm Let $\u_*\in \argmin \mathcal{F}_{h,0}$, $w\in \mathcal{K}^i_h(\u_*)$ then
\begin{equation}
0\,=\,\mathcal{E}^i_h(\u_*,w)'[(\mathbf{0},w)]\,=\,
\frac{h^3}{12}\int_\Omega\!J'(D^{2} w)\cdot D^{2}w\,d\xx+h\int_\Omega\!
J'\left(\mathbb E(\u_{*})\right) : Dw\otimes Dw
\end{equation}
that is $\mathcal{E}^i_h(\u_*,w)=0$ and $\mathcal{E}^{i}_{h,\varepsilon}(\u_*+\varepsilon^2\z_w,w)=\eps^{2}h\min \mathcal{Q}_{w}$.
  }
\end{remark}

\begin{remark}{\rm
In Theorem~\ref{ps}  we have shown that every critical point for ${\mathcal E}^i_h$
of the kind $(\u_*,w)$, with $\u_*\in \argmin \mathcal{F}_{h,0}$ and
$w\in \mathcal{K}^i_h(\u_*)$,
can be approximated (in the strong convergence of
$H^1(\Omega,\R^2)\times H^2(\Omega))$
by uniform Palais-Smale sequences of ${\mathcal E}^i_{h,\varepsilon}$\,. Actually the displacement pair sequence can be chosen explicitly of the kind $(\u_*+\eps^2\zz_w,w)$, say with fixed out-of-plane component and in-plane displacement approximated by an infinitesimal correction tuned by the out-of-plane component.
Nevertheless we cannot expect that every uniform Palais-Smale sequence of ${\mathcal E}^i_{h,\varepsilon}$ is equibounded in
$H^1(\Omega,\R^2)\times \A^{i}$, as we are going to show in the next Counterexample.}
\end{remark}

\begin{counterex}\label{counterex 2.6}(\textbf{a uniform Palais-Smale sequence lacking compactness})\\
{\rm If $\Omega\!=\!(0,a)\!\times\!(0,1)$, $\Gamma\!\equiv\!\partial\Omega$ and
${\mathbf f}_h=\gamma\,{\mathbf e}_2({\mathbf 1}_{(0,a)\times\{0\}})-{\mathbf 1}_{(0,a)\times\{1\}})$,  where $\gamma$ is a suitable constant to be chosen later,
then the unboundedness may develop.
\\So by Theorem~\ref{existence3} (clamped plate), $\forall h>0,\,\forall\varepsilon>0$ there exists
$(\u_\varepsilon, w_\varepsilon)\in\argmin \mathcal{E}^0_{h,\varepsilon}$. Hence $(\u_\varepsilon, w_\varepsilon)$ is a uniform  Palais-Smale sequence for $\mathcal{E}^0_{h,\varepsilon}$, moreover we show below that such a sequence must lack weak compactness in $H^1(\Omega,\R^2)\times H^2(\Omega)$ for big $\gamma$. Indeed, if compactness were true, we would obtain (up to subsequences) that $(\u_\varepsilon, w_\varepsilon)\rightharpoonup(\u,w)\in \argmin{\mathcal E}^0_h$, due to Proposition \ref{Gammalim}.
Eventually we show that $\inf {\mathcal E}^0_h=-\infty$, thus obtaining a contradiction.\\
Actually,
due to Euler equations
\begin{equation}\label{EulerControex}
\int_\Omega J'\big(\E(\uu)\big):\E(\vv) \ =\ \int_{\partial\Omega}\ff \cdot  \v \ =\ -\,\gamma\int_\Omega v_{2,2} \qquad \forall\vv\in H^1(\Omega,\R^2)\,,
\end{equation}
so, for every $\u\in\argmin\mathcal{F}_{h,0}$, $\,J'(\E(\u))=\,-\,\gamma {\bf e}_{2}\otimes {\bf e}_{2}\,$, $\,\uu=2\frac {\gamma\nu}{E}\frac{1+\nu}{1+3\nu}(x_1\mathbf{e}_1+x_2\mathbf{e}_2)+r$, $r\in \mathcal{R}$, and by \eqref{E}
\begin{equation}
\label{EhCounterex}
\mathcal E_{h}^{0}(\u,w)\!=\!\left\{\begin{array}{ll} \!\!\displaystyle\frac{h^3}{12}\!\int_\Omega\!\!J(D^2w)
-\frac{h\gamma}{2}\!\int_{\Omega}\!|w_{,2}|^{2}d\xx,& \!\text{if}\, \u\!\in\!\argmin\mathcal{F}_{h,0}(\cdot,0), w\!\in\!\A^{0}\,,\\
&\\ \!\!
\displaystyle+\infty\  &   \text{otherwise in} \ \ H^{1}(\Omega)\times H^{2}(\Omega)\,.\end{array}\right.
\end{equation}
 Hence, if $\u\in\argmin \F_{h,0}$, $w\in \A^0$, we get
\begin{equation*}
\mathcal E_{h}^{0}(\u,w)
\leq
\frac{C_\nu\,E\,h^3}{24}\int_\Omega\!\vert D^2w\vert^2d\xx
-\frac{h}{2}\gamma\int_\Omega\!|w_{,2}|^{2}\,d\xx.
\end{equation*}
Set $w(x_1,x_2)=\alpha(x_1)\beta(x_2)$, with
 $\alpha\in H^2_0(0,a)$ and $\beta\in H^2_0(0,1)$. Then $w\in H^2_0(\Omega)$ and
\begin{equation}
\label{boundE}
\mathcal E_{h}^{2}(\u,w)\leq \left( A_0C_0+A_1C_1+ A_2\right)C_\nu\,\frac{Eh^3}{24}\int_0^1\vert\beta''\vert^2\,dx_{2}- \frac{A_2h\gamma}{2}\int_0^1\vert\beta'\vert^2dx_{2},
\end{equation}
where
$$A_0= \int_0^1\vert \alpha''\vert^2dx_{1},\:\:\:A^1= 2\int_0^1\vert \alpha'\vert^2dx_{1},\:\:\:A_2= \int_0^1 \alpha^2dx_{1}$$
and $C_0,C_1$ are the best constants such that
$$\int_{0}^1 \beta^2dx_{2}\leq C_0\int_{0}^1\vert \beta''\vert^2dx_{2},\:\:\:\int_{0}^1\vert \beta'\vert^2dx_{2}\leq C_1\int_{0}^1\vert \beta''\vert^2dx_{2}\:\:\:\:\forall \beta\in H^2_0(0,1)$$
If $\xi\in H^2_0(0,1)$ is the eigenfunction fulfilling the equality
$\int_{0}^1\vert \xi'\vert^2dx_{2} = C_1\int_{0}^1\vert \xi''\vert^2dx_{2}$ and
$$\gamma > \frac{1}{6}\left( A_0C_0+A_1C_1+ A_2\right)\,C_\nu\,E\,h^2/( A_2C_1)\,.$$
Setting $\beta_n\!:=\!n\xi\!\in \!H^2_0(0,1)$ and $w\!=\!\alpha\beta_{n}$, the right-hand side of \eqref{boundE} goes to $-\infty$ as $n\!\rightarrow \!\infty.$
}
\vspace{-0.1cm}
\end{counterex}
In the previous counterexample we have shown that some uniform Palais-Smale sequence may be not converging to any critical point, while in the next examples we show how Theorem 2.3 can be used to detect buckled configurations of the plate (associated to critical points for \textbf{FvK}) by means of uniform Palais-Smale
sequences for the approximating functionals.
\begin{example}
\label{example 2.7} (\textbf{buckling of a rectangular plate under compressive load})\\
{\rm Set $\Omega  = (0,a)\times(0,1)$,
${\mathbf f}_h=\gamma\,{\mathbf e}_2({\mathbf 1}_{(0,a)\times\{0\}})-{\mathbf 1}_{(0,a)\times\{1\}})$ and $\Gamma =\Sigma_+\cup \Sigma_-$, with \\ $\Sigma_+=[0,1]\times\{1\}$, $\Sigma_-=[0,1]\times\{0\}$.\\
Now $\displaystyle\Gamma\neq\partial\Omega$\,: by arguing as the in previous Counterexample we find
noncompact uniform Palais-Smale sequences together with energy of admissible configurations unbounded from below.
\\
In the present case we push forward the analysis: as before we find that if $\u\in \argmin \F_{h,0}$ and $w\in \A^i$, i=0,1,2, then $J'(\mathbb E(\u))=-\gamma {\mathbf e}_2\otimes{\mathbf e}_2$, so that
$${\mathcal E}^i_h(\u, w)=\frac{h^3}{12}\int_\Omega\!J(D^2w)d\xx-\frac{h\gamma}{2}\int_\Omega\!|w_{,2}|^2d\xx
\qquad
\text{if}\ \u\in\argmin\mathcal{F}^i_{h,0}(\cdot,0),\ w\in\A^{i}\,.$$
	We look for critical points in the form  $w=w(x_2)$ under the following conditions:
	\begin{itemize}
    \item[]$w(0)=w(1)=w'(0)=w'(1)=0$, \quad\qquad\ \ if $i=0$\,;
	\item[]$w(0)=w(1)=0$, \qquad\qquad\qquad\qquad\qquad\quad if $i=1$\,;
	\item[]$w(0)''=w''(1)=w'''(0)=w'''(1)=0$, \qquad if $i=2$\,.
	\end{itemize}
Since $J({\mathbf e}_2\otimes{\mathbf e}_2)=\frac{E}{2(1-\nu^2)}$, we have
$${\mathcal E}^i_h(\u, w)=\frac {Eh^3}{24(1-\nu^2)}
\int_0^1\vert w''(x_2)\vert^2\,dx_{2}-\frac{h\gamma\,a}{2}\int_0^1\vert w'(x_2)\vert^{2}\,dx_{2}$$
whose non-trivial  critical points can be easily computed, via the ODE
$$w''''+\frac{12\gamma a (1-\nu^2)}{Eh^2}w''=0.$$
Theorem~\ref{ps} allows to recover Palais-Smale sequences for $\mathcal E^i_{h,\varepsilon}\,, \ i=0,1,2$.\\
In the clamped case ($i=0$) the nontrivial buckled solutions occur for discrete choices of $h$\,:
$$
h_n\, =\, \frac 1 {2\,n\,\pi}\,\sqrt{\frac {12\,\gamma\,a\,(1-\nu^2)}{E}}\,,\quad
w_n(x_2)\ =\ 1\,+\, \sin \left(\,\sqrt{\frac {12\,\gamma\,a\,(1-\nu^2)}{E}}\,
\frac 1{\,h}\, (x_2-\pi/2)\right), \quad n\in\N;
$$
else, for any other choice of $h$, $w\equiv 0$\,.\\
The associated Palais-Smale sequence is $\big(\,2\frac {\gamma\nu}{E}\frac{1+\nu}{1+3\nu}(x_1\mathbf{e}_1\!+\!x_2\mathbf{e}_2)+\varepsilon^2 \zz_{w_n}(x_1,x_2)\,,\,w_n(x_2)\big)$\,, where
$\zz_{w_n}(x_1,x_2)=\Big(\,0\,,\,1/2\,\int_0^{x_2}|w_n'(t)|^2dt\,\Big)$ and $w_n$ is given above.}
\end{example}
\begin{example}
\label{example 2.8}(\textbf{buckling of a rectangular plate under shear load}).\\
{\rm Set $ \Omega=(-2,2)\times(-1,1)$, $i=0\,,$ and $\Gamma=\Sigma^{1,\pm}\cup \Sigma^{2,\pm}$, where:\\
$\Sigma^{1,+}=[-2,0]\times\{1\}$, $\Sigma^{1,-}=[0,2]\times\{-1\}$,
$\Sigma^{2,+}=\{2\}\times[-1,1]$, $\Sigma^{2,-}=\{-2\}\times[-1,1]$.\\
Assume $\mathbf{f}_h=\gamma \ttau\big(\mathbf{1}_{S^{2,\pm}}-\mathbf{1}_{S^{1,\pm}}\big)$, where
$S^{2,\pm}=\Sigma^{2,\pm}$, $S^{1,\pm}=[-2,2]\times\{\pm 1\}$, $\gamma>0$, $\ttau$ is the counterclockwise oriented  tangent unit vector to $\partial\Omega=S^{1,\pm}\cup S^{2,\pm}$.
\\
Since $\u\in \argmin \F_{h,0}$, by exploiting Euler-Lagrange equations as before, we obtain\\
$J'(\mathbb E(\u))=\gamma ({\mathbf e}_1\otimes{\mathbf e}_2+ {\mathbf e}_2\otimes{\mathbf e}_1)$
and by \eqref{E}
$${\mathcal E}^0_h(\u, w)=\frac{ h^3}{12}\int_\Omega\!J(D^2w)dx+h\gamma\int_\Omega\!w_{,1}\,w_{,2}\,d\xx.$$
We look for critical points in the form
\begin{equation}
w=\left\{
\begin{array}{ll}
\psi(x_1-x_2)
 \quad &\hbox{ if }\:\: (x_1,x_2)\in \Omega,\:\:\:\vert x_1-x_2\vert\leq 1\\
\\
0
& \quad \hbox{ else in } \Omega,
\end{array} \right.
\end{equation}
and satisfying $\psi(\pm 1)=\psi'(\pm 1)=0$.\\
By
$\,J({\mathbf e}_1\otimes{\mathbf e}_1+
{\mathbf e}_2\otimes{\mathbf e}_2-{\mathbf e}_1\otimes{\mathbf e}_2-{\mathbf e}_2\otimes{\mathbf e}_1)
=\displaystyle \frac{2E}{1-\nu^2}\,$ we obtain
$${\mathcal E}^0_h(\u, w)=\frac{h^3E}{3(1-\nu^2)}
\int_{-1}^1\vert \psi''(t)\vert^2\,dt-2h\gamma\,\int_{-1}^1\vert \psi'(t)\vert^2\,dt$$
whose nontrivial critical points can be easily computed, via the ODE
$$\psi''''+\frac{6\gamma(1-\nu^2)}{Eh^2}\,\psi''=0\,, \qquad
\psi(\pm 1)=\psi'(\pm1)=0\,.$$
Therefore even now the nontrivial buckled solutions occur for (different) discrete choices of $h$\,:
$$
w\,=\,w_n(x_1,x_2)\ =\ \psi_n (x_1-x_2)\ :=\ 1\,+\, \sin \left(\,\sqrt{\frac {12\,\gamma\,a\,(1-\nu^2)}{E}}\,
\frac 1{\,h_n}\, (x_1-x_2+1/2)\right)
$$
$$ \hbox{if }\qquad
h_n\, =\, \frac 1 {n\,\pi}\,\sqrt{\frac {12\,\gamma\,a\,(1-\nu^2)}{E}}\,,\qquad
\quad \hbox{with}\quad n\in\N;
$$
else, we have the flat solution \,$w\equiv 0$\,  for any other choice of $h$ \,.\\
The associated Palais-Smale sequence is \,$\big(\,\uu(x_1,x_2)+\varepsilon^2 \zz_{w_n}(x_1,x_2)\,,\,w_n(x_1,x_2)\,\big)$\,, where\\
\centerline{$ \uu(x_1,x_2)\!=\! \gamma\frac{1+\nu}{E}(x_2,x_1)\,,\
\zz_{w_n}(x_1,x_2)\!=\! \!\Big(\!-(1/2)\! \int_{-1}^{x_{1}-x_{2}} \!|w_{n}'(t)|^{2}\, dt\,, \ (1/2)\!\int_{-1}^{x_{1}-x_{2}} |w_{n}'(t)|^{2}\,dt\,\Big)\, .
$}
}
\end{example}
\begin{remark}\label{rmk2.9}
{\rm In Examples  \ref{example 2.7}, \ref{example 2.8}, when nontrivial solutions exist the period of the oscillations has order $h$. By scaling loads, that is by
taking ${\bf f}_{h}\!=\!h^\alpha \mathbf f$, we get
$J'(\mathbb E(\u))=-h^\alpha\gamma( {\mathbf e}_2\otimes{\mathbf e}_2)$ and
$J'(\mathbb E(\u))=h^\alpha\gamma ({\mathbf e}_1\otimes{\mathbf e}_2+ {\mathbf e}_2\otimes{\mathbf e}_1)$ respectively, while related limit functionals become respectively
$${\mathcal E}^i_h(\u, w)=\frac {Eh^3}{24(1-\nu^2)}
\int_0^1\vert w''(x_2)\vert^2\,dx_{2}-\frac{h^{\alpha+1}\gamma\,a}{2}\int_0^1\vert w'(x_2)\vert\,dx_{2}\,,\qquad i=0,1,2,$$
$${\mathcal E}^0_h(\u, w)=\frac{h^3E}{3(1-\nu^2)}
\int_{-1}^1\vert w''(t)\vert^2\,dt-2h^{\alpha+1}\gamma\,\int_{-1}^1\vert w'(t)\vert^2\,dt,$$
whose nontrivial critical points obviously exhibit  oscillation period of order $h^{1-\alpha/2}$.}
\end{remark}
Computations in Remark \ref{rmk2.9} proves useful in the next Section when studying
asymptotics of the problem as the thickness tends to $0_+$\,.





\section{Scaling F\"oppl-von K\'arm\'an energy
}
\label{section hscaling}
Here we focus on the asymptotic analysis of the mechanical problems for $\mathbf{Fvk}$ plate as $h\to0_+$. To highlight properties of the limit solution we examine the behavior of suitably scaled energy:
all along this Section we assume that there is no transverse load, say $g_{h}\equiv 0$,
while we refer to a parameter $\alpha$ characterizing different asymptotic regimes of in-plane load $\mathbf{f}_h$, say
\begin{equation}\label{fh=halfaf}
  {\bf f}_{h}=h^{\alpha}{\bf f} \ \hbox{ where } \ \alpha\ge 0 \ \hbox{  and } \ {\bf f}\in L^{2}(\partial\Omega, \mathbb R^{2})\,.
\end{equation}
The next result and the subsequent counterexample show
how the choice of $\alpha$ may influence the asymptotic behavior of functionals $\mathcal F_{h}$ when $h\to 0_{+}$.
\begin{theorem}\label{scaling1} Let $\Omega\subset \R^2$ be a bounded connected Lipschitz open set, $\alpha\ge 2$ and $ i=0,1,2$.\\
If $i=0$ (clamped plate) assume \eqref{equil}  and $\Gamma=\partial\Omega$ (as in Theorem \ref{existence3}) .\\
If $i=1$ (simply supported plate) assume \eqref{equil}, $\Omega$\,strictly convex,
$\Gamma\!=\!\partial\Omega$ (as in Theorem \ref{existence1}).\\
If $i=2$ (free plate) assume \eqref{nullres} and \eqref{tractionload}  (as in Theorem\,\ref{thm global traction}).\\
Set
\\
\begin{equation}\label{gammascaling}
{\mathcal F}^{i,\alpha}(\vv,\zeta)=\left\{\begin{array}{ll}& {\mathcal F}^{i}_{1}(\vv,\zeta)\ \ \hbox{if} \ \ \alpha=2\\
&\\
& {\mathcal F}^{i}_{1}(\vv,\zeta)+\chi_{\{D^{2}\zeta\equiv 0\}}(\zeta)\ \ \hbox{if} \ \ \alpha>2\,,\\
\end{array}\right.
\end{equation}
where $\chi_{\{D^{2}\zeta\equiv 0\}}(\zeta)=0$ if $D^{2}\zeta\equiv 0$, $=+\infty$ else.\\
Fix $i\in\{0,1,2\}$ and a sequence $(\u_{h},w_{h})$ in $ \argmin\F_{h}^{i}$.\\
Then there exists
$(\v,\zeta)\in \argmin {\mathcal F}^{i,\alpha}$ such that, up to subsequences,
\begin{equation}
(h^{-\alpha}\u_{h},h^{-\alpha/2}w_{h})\to (\v,\zeta)
\quad \hbox{ weakly in }
H^{1}(\Omega,\mathbb R^{2})\times H^2(\Omega)\,,\hbox{ as }h\to0_+\,.
\end{equation}
Moreover
\begin{equation}
h^{-2\alpha-1}\mathcal F_{h}^{i}(\u_{h},w_{h})\to {\mathcal F}^{i,\alpha}(\vv,\zeta)
\,,\hbox{ as }h\to0_+\,.\end{equation}
\end{theorem}
\begin{proof}
The case $\alpha=2$ is trivial since $(\u_{h},w_{h})\in \argmin {\mathcal F}^{i}_{h}$ if and only if $(h^{-2}\u_{h},h^{-1}w_{h})\in\argmin {\mathcal F}^{i}_{1}$
for every $ h $.\\
If $\alpha> 2$, $i=0,1$ and $(\u_{h},w_{h})\in \argmin {\mathcal F}^{i}_{h}$, set $\v_{h}:=h^{-\alpha}\u_{h},\ \ \zeta_{h}:=h^{-\alpha/2}w_{h},\ \  \lambda_{h}=\|\mathbb E(\v_{h}\|_{L^{2}}$ and assume by contradiction $\lambda_{h}\to +\infty$. Then by taking into account minimality of $(\u_{h},w_{h})$, \eqref{coerc}, \eqref{equil} and setting $\varphi_{h}=\lambda_{h}^{-1/2}\zeta_{h},\ {\bf z}_{h}=\lambda_{h}^{-1}\v_{h}$ we get
\begin{equation}\label{scalminseq1}
\displaystyle c_\nu\frac{h^{2-\alpha} \,E}{24}\int_\Omega\vert D^2 \varphi_{h}\vert^2+
\lambda_{h}\,c_\nu\,\frac{E }{2}\int_\Omega\,|{\mathbb{D}}({\bf z}_{h},\varphi_{h})|^{2}
\ \leq \ \int_{\partial\Omega}{\bf f}\cdot {\bf z}_{h} \le C.
\end{equation}
Hence $|D^{2}\varphi_{h}|\to 0$ in $L^2(\Omega,{\rm Sym}_{2,2}(\R))$ and by  taking into account that $\varphi_{h}=0$ on $\partial\Omega$ we get $\varphi_{h}\to 0$ in $H^{2}(\Omega)$; therefore $\mathbb E({\bf z}_{h})\to \mathbb{O}$ in $L^{2}(\Omega,\mathbb R^{2})$, a contradiction since $\|\mathbb E({\bf z}_{h})\|_{L^{2}}=1$. Then $\lambda_{h}$ is bounded from above and by taking into account minimality of $(\u_{h},w_{h})$ , \eqref{coerc}, \eqref{equil} we get
\begin{equation}
\displaystyle c_\nu\frac{h^{2-\alpha} \,E}{24}\int_\Omega\vert D^2 \zeta_{h}\vert^2+
c_\nu\,\frac{E }{2}\int_\Omega\,|{\mathbb{D}}(\vv_{h},\zeta_{h})|^{2}
\ \leq \ \int_{\partial\Omega}{\bf f}\cdot {\bf v}_{h} \le \|{\bf f}\|\lambda_{h}\le C
\end{equation}
which entails $D^{2}\zeta_{h}\to 0$ in $L^{2}(\Omega)$ and equiboundedness of $D\zeta_{h}$ in $L^{4}(\Omega, \mathbb R^{2})$. \\
When $i=2$ we take again  $\lambda_{h}=\|\mathbb E(\v_{h})\|_{L^{2}}$ and assume  by contradiction $\lambda_{h}\to +\infty$. Then estimate \eqref{scalminseq1} continues to hold and as before $|D^{2}\varphi_{h}|\to 0$ in $L^{2}(\Omega)$ which entails $\varphi_{h}-\,-\! \!\!\!\!\!\int_{\Omega}\varphi_{h}\to 0$ in $L^{2}(\Omega)$,
$D\varphi_h\to \mathbf{c}$ in $L^4$
and $2\mathbb E({\bf z}_{h})\to -{\bf c}\otimes {\bf c}$ strongly in $L^2(\Omega,{\rm Sym}_{2,2}(\R))$  for a suitable ${\bf c}\in \mathbb R^{2}$. Therefore  \eqref{tractionload}, \eqref{scalminseq1} yield
\begin{equation}
0\le \displaystyle \lim_{h\to 0_{+}} \int_{\partial\Omega}{\bf f}\cdot {\bf z}_{h}=\lim_{h\to 0_{+}} f\int_{\partial\Omega}{\bf n}\cdot {\bf z}_{h}=\lim_{h\to 0_{+}} f\int_{\Omega}\hbox{\rm div}\,{\bf z}_{h}=-\frac{f}{2}|\Omega||{\bf c}|^{2}
\end{equation}
that is ${\bf c}={\bf 0}$ so $\mathbb E({\bf z}_{h})\to \mathbb{O}$ in $L^2(\Omega,{\rm Sym}_{2,2}(\R))$ as in the previous cases, again a contradiction. Thus equiboundedness holds in this case too.
Since, for $0<h\leq 1$, the w.l.s.c. functionals $\mathcal{F}^{i,\alpha}$ fulfil $\mathcal{F}^{i,\alpha}\leq
h^{-2\alpha-1}\mathcal{F}^{i}_h$,
the proof can be completed by a standard argument in $\Gamma$ convergence.
\end{proof}
\begin{remark}
{\rm It is worth noticing that when $D^{2}w\equiv 0$ then
\begin{equation}
\F_{1}(\v,w)=\F_{1}(\v,{0})\,, \qquad \ \hbox{ if }i=0,1\,,
\end{equation}
\begin{equation}
\F_{1}(\v,w)=\F_{1}(\v,\boldsymbol\xi\cdot\xx)=\int_{\Omega}J(\E(\v)+\frac{1}{2}{\boldsymbol \xi}\otimes{\boldsymbol \xi})-\!\int_{\partial\Omega}\!{\bf f}_h\cdot\v\:\:\:
\hbox{ for}\:\:\:w=\boldsymbol\xi\cdot\xx,\
\hbox{ if }i=2.
\end{equation}
}
\end{remark}
Theorem \ref{scaling1} is optimal in the sense that if $\alpha<2$ we cannot expect neither that $h^{-2\alpha-1}\min_{\mathcal A^{i}}\mathcal F_{h}$ are bounded from below nor that minimizers are equibounded in $H^{1}(\Omega, \mathbb R^{2})\times H^{1,4}(\Omega)$ when we let $h\to 0_+$. This phenomenon may take place even if $\Omega$ is a rectangle as shown by the  next Counterexample,
where we consider a plate with the same geometry and load of Counterexample
\ref{counterex 2.6} 
, nevertheless here we push further the analysis of this case.\\

\begin{counterex}
\label{infmenoinfinito}
{\rm
Let $a> EC_{\nu},\ \alpha\in [0,2),\ \mathbf f_{h}=h^{\alpha}\mathbf f$ with
 \begin{equation}\label{OMEGAcontroes}
    \Omega \, =\, (0,a)\times(0,1)\,,
    \quad\Gamma\,=\,\partial\Omega\,,
    \quad
    \ g_{h}\,\equiv 0\,,\quad
     \mathbf{f}=\big(\mathbf{1}_{\{y=0\}}-\mathbf{1}_{\{y=1\}})\, \mathbf{e}_2\,.      \end{equation}
Then for any sequence $(\u_{h},w_{h})\in\arg\min\F_{h}^{0}$
(such sequences do exist due to Theorem \ref{existence3}), the scaled sequence $(h^{-\alpha}\u_{h},h^{-\alpha/2}w_{h})$
is not  equibounded in $H^{1}(\Omega,\mathbb R^{2})\times H^{1,4}(\Omega)$\,.
Moreover, $\,\inf\,h^{-2\alpha-1}\!\F_{h}^0\to -\infty$ as $h\to 0_+$.

\vskip0.1cm
Indeed  we can set: $\ \v_{h}:=h^{-\alpha}\u_{h}\,,\ \ \zeta_{h}:=h^{-\alpha/2}w_{h}\,,\ $ and
 \begin{equation}\label{fscaled}
\mathcal W_{h}(\v_{h},\zeta_{h})\,:=\,
h^{-1-2\alpha}\F_{h}(\u_{h},w_{h})
\,=\,\frac{h^{2-\alpha}}{12}\int_{\Omega}J(D^{2}\zeta_{h})+\int_{\Omega}J({\mathbb D}(\v_{h},\zeta_{h}))-\int_{\partial\Omega}{\bf f}\cdot\v_{h} \,,
\end{equation}
\begin{equation}
  {\mathcal{I}}^+(\v,\zeta)\ :=\ \inf\, \left\{\, \limsup_{h\,\to\,0_+}\ \mathcal W_{h}(\v_{h},\zeta_{h})\,:\
  \vv_h \mathop{\rightharpoonup}^{w-H^1}\v \,, \ \zeta_h \mathop{\rightharpoonup}^{w-H^{1,4}}\zeta \,\right\}\ ,
\end{equation}
\begin{equation}
  {\mathcal{I}}^-(\v,\zeta)\ :=\ \inf\, \left\{\, \liminf_{h\,\to\,0_+}\  \mathcal W_{h}(\v_{h},\zeta_{h})\,:\
  \vv_h \mathop{\rightharpoonup}^{w-H^1}\v \,, \ \zeta_h \mathop{\rightharpoonup}^{w-H^{1,4}}\zeta \,\right\},
\end{equation}
\begin{equation} \mathcal J(\mathbb{B},\boldsymbol{\eta})=\dfrac{E}{8(1+\nu)}\left|\mathbb{B}+\mathbb{B}^T+ \boldsymbol{\eta} \otimes \boldsymbol{\eta}\right|^2\,+\,
\dfrac{E\nu}{8(1-\nu^{2})}\left| \Tr \left(\mathbb{B}+\mathbb{B}^T+ \boldsymbol{\eta} \otimes \boldsymbol{\eta}\right)\right|^2 \ .\end{equation}
Then by arguing as  in Lemma 4.1 of \cite{BD}   we get
\begin{equation}\label{I+leq}
  \mathcal I^{+}(\v,\zeta)\ \le\ \Lambda(\v,\zeta)\ :=\ \int_\Omega\,{\mathcal J}({\mathbb{D}}(\v,\zeta))\,dx-\!\!\int_{\partial\Omega}
\!\!{\bf f}\cdot\v\,d\H^1.
\end{equation}
Then by  denoting with $Q\mathcal J$ the quasiconvex envelope of $\mathcal J$, since $\mathcal{I}^+$ is sequentially lower semicontinuous in $w-H^1\times w-H^{1,4}$,
we obtain
\begin{equation}\label{I+leq bis}
   \mathcal{I}^+(\v,\zeta)\ \leq \
   \int_\Omega\,Q\mathcal J\left(D\v,D\zeta \right)
    \, dx\,-\,
    \int_{\partial\Omega}\, \mathbf{f}\cdot \v    \, d\xx.
\end{equation}
On the other hand for every $ \displaystyle\v_h \,\mathop{\rightharpoonup}^{w-H^1}\,\v \,, \ \zeta_h \,\mathop{\rightharpoonup}^{w-H^{1,4}} \, \zeta$ we get
$$ \liminf_{h\,\to\,0_+}\ h^{-1-2\alpha}\,\mathcal{F}_h(h^{\alpha}\v_h,h^{\alpha/2}\zeta_h)\ge \int_\Omega Q{\mathcal J}(D\v,D\zeta) \,-\, \int_{\partial\Omega}{\bf f}\cdot\v$$
that is
\begin{eqnarray*}
{\mathcal{I}}^{-}(\v,\zeta)   \ge \int_\Omega Q{\mathcal J}(D\v,D\zeta) \,-\, \int_{\partial\Omega}{\bf f}\cdot\v.
\end{eqnarray*}
By
\begin{equation}\label{glim}
 {\mathcal{I}}(\v,\zeta)\ :=\ \int_\Omega Q\mathcal J(D\v,D\zeta) \,-\, \int_{\partial\Omega}{\bf f}\cdot\v
\ \geq \
 {\mathcal{I}}^+(\v,\zeta)\ \geq \ {\mathcal{I}}^-(\uu,w) \ \geq \
{\mathcal{I}}(\v,\zeta)
\end{equation}
we get
\begin{equation}\label{Gconv}
 \mathop{\Gamma\lim}_{h\to 0_+}\ \mathcal{W}_h\ = \ \mathcal{I}\ .
\end{equation}
%
Therefore, if  $(h^{-\alpha}\u_{h}^{*},h^{-\alpha/2}w_{h}^{*})$ were equibounded in $H^{1}(\Omega,\mathbb R^{2})\times H^{1,4}(\Omega)$ then
$$h^{-1-2\alpha}\mathcal F_{h}(\u_{h}^{*},w_{h}^{*})\to \min  {\mathcal I} = \inf \Lambda$$
since $\Lambda$ is the relaxed functional of $\mathcal{I}$, and we will show that this leads to a contradiction.
Indeed, we choose
   \begin{equation}\label{wn}
    \zeta_n(x,y)\,=\, 
   \dfrac{1}{\sqrt n} \varphi(ny)\,\psi_n(x), \quad \v_n(x,y)\,=\, (0,\dfrac{-n}{2}y)\,,
  \end{equation}
   with
  \begin{equation}\label{varphi}
    \varphi :\R\to\R \,, \ \ 1\,\hbox{\rm -periodic}\,, \varphi(y)=\dfrac{1}{2}(1-|1-2y|)\ \forall y\in (0,1)
   \end{equation}
    \begin{equation}\label{psi}
    \psi_n(x)\,=\,
    nx \mathbf{1}_{\{[0,1/n]\}}
    + \mathbf{1}_{\{[1/n,a-1/n]\}}
    -n(x-a) \mathbf{1}_{\{[a-1/n,a]\}}\,.
      \end{equation}
 We get
  \begin{equation*}
  \E(\v_n)\! =\! \left[
  \begin{array}{cc}
     0 &  \ 0 \\
     0 &  -\dfrac{n}{2}
  \end{array}
	  \right] \ , \quad \mathbb{D}(\v_{n},\zeta_{n})\! =\!
  \left[
  \begin{array}{cc}
     \dfrac{1}{2n}\big (\psi_n'(x)\big)^2\big(\varphi(ny)\big)^2\ \ &  \  \dfrac{1}{2}\psi_n(x) \psi_n'(x)\varphi(ny)\varphi'(ny)\\
     \dfrac{1}{2}\psi_n(x) \psi_n'(x)\varphi(ny)\varphi'(ny)\ \  &  \dfrac{n}{2}\big(\psi_n^{2}(x)|\varphi'(ny)|^{2}-1\big )
  \end{array}
  \right]
\end{equation*}
and by taking into account \eqref{density}, \eqref{coerc}
and that $2|\varphi|\le 1,\ |\varphi'|=1,\ \ |\psi|\le 1,\ \ |\psi_{n}'|\le n, \ \spt \psi_{n}'\subset [0,1/n]\cup [a-1/n,a],\ \ |\psi_{n}|=1$ on $[1/n,a-1/n]$, $a>EC_{\nu}$,
\begin{equation*}
\hskip-1cm
\begin{array}{ll}
&\displaystyle\!
\Lambda(\v_n,\zeta_n)\ =\
\int_0^a\!\int_0^1
{J}(\,\mathbb{D}(\v_n,\zeta_n)\,)\, \, dx\,dy
    \,-\,
    \int_{\partial\Omega}\, \mathbf{f}\cdot \v_n    \, \, dx\,dy\ \le\\
    &\\
 &\displaystyle \le\int_{0}^{a}\!\!\int_{0}^{1}\!\dfrac {EC_{\nu}}{8}\left (\,n^{-2}|\psi_n'(x)|^4|\varphi(ny)|^4\!+\!
      2|\psi_n(x)|^{2}| \psi_n'(x)|^{2}|\varphi(ny)|^{2}|\varphi'(ny)|^{2}\!+\!
			n^2\left ( \psi_n^{2}(x)|\varphi'(ny)|^{2}\!-\!1\right)^2
     \,\right)\!-\!\dfrac{na}{2}
    \\
    &\\
    &\displaystyle\le\!\int_0^a\!\!\int_0^1\!
    \dfrac {EC_{\nu}}{8}\!\left(n^{2}{\bf 1}_{[0,1/n]\cup [a-1/n,a]}+ n^2\left( \psi_n^{2}(x)-1\right)^2\right) \,-\,
    \dfrac{na}{2}\le \dfrac {nEC_{\nu}}{2}\,-\,
    \dfrac{na}{2}\ \to\  -\infty \ .
      \end{array}
\end{equation*}
leads to a contradiction. \\
So $(h^{-\alpha}\u_{h}^{*},h^{-\alpha/2}w_{h}^{*})$ are not equibounded in $H^{1}(\Omega,\mathbb R^{2})\times H^{1,4}(\Omega)$ and the first claim follows.\\
Eventually we prove the second claim. By \eqref{I+leq} there exists $(\v_{n,h},\zeta_{n,h})\to (\v_{n},\zeta_{n})$ weakly in $H^{1}(\Omega,\mathbb R^{2})\times H^{2}(\Omega)$ such that $\limsup \mathcal W_{h}(\v_{n,h},\zeta_{n,h})\le \mathcal I(\v_{n},\zeta_{n})\le -Kn$ for suitable $K>0$, hence by using a diagonal argument we achieve the claim.}
 \end{counterex}
\begin{remark}\label{unbdclamp}{\rm If $a> EC_{\nu},\ \alpha\in [0,2),\ \mathbf f_{h}=h^{\alpha}\mathbf f$ with
 \begin{equation}\label{OMEGAREM}
    \Omega \ =\ (0,a)\times(0,1)\,,\quad \ g_{h}\equiv 0,\
     \mathbf{f}=\big(\mathbf{1}_{\{y=0\}}-\mathbf{1}_{\{y=1\}})\, \mathbf{e}_2,\ \ \Gamma=\partial\Omega\,,
     \end{equation}
Then $h^{-1-2\alpha}\inf\mathcal F_{h}^{i}\to -\infty$ as $h\to 0_{+}$} holds true also for $i=1,2$.\\
Indeed, though existence of  minimizers  of  ${\mathcal F}_{h}^{i},\  (i=1,2)$ may fail, nevertheless $\inf\mathcal F_{h}^{i}\le \inf\mathcal F_{h}^{0}$ for $i=1,2$; hence the claim follows by previous Counterexample.
\end{remark}
\section{Prestressed plates: oscillating versus flat equilibria.}
\label{OscillSect}
Counterexample \ref{infmenoinfinito} and Remark 3.4 show that the F\"oppl Von Karman functional might not be suitable for studying equilibria of plates when thickness $h\to 0_{+}$, at least in presence of in-plane loads scaling as $h^{\alpha}$, when $ \alpha\in [0,2))$ and $h$  is the scale factor for the plate thickness.\\ To circumvent this difficulty, as in the case of many practical engineering applications, we assume that our plate-like structure is initially prestressed and undergoes
a transverse displacement about the prestressed state.\\
Precisely, in this Section we fix  $g_{h}\equiv 0$,  $\mathbf{f}\in L^{2}(\partial \Omega,\mathbb R^{2})$, $\alpha\in [0,2)$ and  we assume that the prestressed state is caused by the (scaled) force field $\mathbf{f}_{h}=h^{\alpha}{\bf f}$ and is given by every $\u^{*}\in H^{1}(\Omega,\mathbb R^{2}),\ \u^{*}=h^{\alpha}\v^{*}$ where $\v^{*}$ is a minimizer of the  functional
\begin{equation}
\F(\v):=\int_{\Omega}J(\mathbb E(\v))-\int_{\partial\Omega}{\bf f}\cdot\v
\end{equation}
The transverse displacement $w$ is chosen such that the pair $(\u^{*},w)$ minimizes
the functional ${\mathcal G}_{h}$ over $ H^{1}(\Omega,\mathbb R^{2})\times \mathcal{A}^i$, defined by
 \begin{equation*}
\displaystyle {\mathcal G}_{h}({\bf u}, w)\!=\!\left\{\begin{array}{ll}\!\! \displaystyle \F_{h}({\bf u}, w)&\!\! \text{ if}\ \u=\u^{*} \hbox{ and } w\in \mathcal{A}^i,\  \\
\!\!\displaystyle+\infty\  & \!\! \text{ else}\,.\end{array}\right.
\end{equation*}
Moreover  we have ${\mathcal G}_{h}({\bf u}, w)= \widetilde{\mathcal G}_{h}({\bf v}, \zeta)$ when setting
$\v:= h^{-\alpha}\u,\ \zeta:={h}^{-\alpha/2}w$  and
\begin{equation*}
\displaystyle \widetilde{\mathcal G}_{h}({\bf v}, \zeta)=\left\{\begin{array}{ll} \displaystyle h^{\alpha}F_{h}^{b}(\zeta)\,+\,h^{2\alpha+1}\!\int_{\Omega }\,J({\mathbb D}(\v,\zeta))\, -\, h^{2\alpha+1}\!\int_{\partial \Omega} \mathbf{f}\cdot\v\,,&\ \text{if}\ \v\in\argmin\F,\ \zeta\in\mathcal{A}^i\,,\\
&\\
\displaystyle+\infty\  & \ \ \text{else in} \ \ H^{1}(\Omega)\times \mathcal{A}^i \,. \end{array}\right.
\end{equation*}
We aim to capture the nature of the transverse minimizer through a detailed study of the asymptotic behavior of  minimizers of $\widetilde{\mathcal G}_{h}$ as $h\to 0_{+}$.
A first hint in this perspective is the next result.
\begin{theorem} \label{relax} For every  $\v\in\arg\min\mathcal F$, let
${I}^{**}_{\vv}(\xx,\cdot)$ be the convex envelope of ${I}_{\vv}(\xx,.)$  where ${I}_{\vv}(\xx,\boldsymbol\xi):=J\left (\mathbb E(\v)(\xx)+\dfrac{1}{2}\boldsymbol\xi\otimes\boldsymbol\xi\right )$,  and
\begin{equation}\label{G00}\mathcal G^{**}(\vv,\zeta):=\int_{\Omega }{I}^{**}_{\vv}(\xx,D \zeta)\,d\xx- \int_{\partial \Omega} \mathbf{f}\cdot\v\, d\mathcal H^{1}
\qquad \forall\,\zeta\in H^{1,4}(\Omega)\,.\end{equation}
 Then, for every $\alpha\in [0,2)$,
\begin{equation}\label{tildeGrescaled}
h^{-2\alpha-1}\min _{\mathcal A^{i}}\widetilde{\mathcal G}_{h}\to \left\{\begin{array}{ll}&\min \left\{\mathcal G^{**}(\vv,\zeta):\ \zeta\in H^{1,4}(\Omega),\ \zeta=0\ \text{on}\ \Gamma\right\}\ \ \hbox{if}\ \ i=0,1\\
&\\
&\min\{\mathcal G^{**}(\vv,\zeta):\ \zeta\in H^{1,4}(\Omega)\}\ \ \hbox{if}\ \ i=2\,.\\
\end{array}
\right.
\end{equation}
Moreover if $(\v,\zeta_{h})\in \arg\min_{\mathcal A^{i}}\widetilde{\mathcal G}_{h}$ then $\zeta_{h}\to \zeta $ weakly in $H^{1,4}(\Omega)$, up to subsequences,  with $(\vv,\zeta)\in\arg\min\mathcal G^{**}$.
\end{theorem}

\begin{proof} The claim is a straightforward consequence of techniques developed in Lemma 4.1 of \cite{BD}  and standard relaxation of integral functionals.
\end{proof}
In order to characterize  equilibrium configurations of $\widetilde{\mathcal G}_{h}$,
additional information about minimizers of functional $\mathcal G^{**}$ are needed:
actually
a careful use of Theorem \ref{relax} allows to show explicit examples capturing the qualitative behavior of minimizers and their dependance on the thickness $h$.\\
To this aim,
if $\mathbb{A}\!\in\! {\rm Sym}_{2,2}(\R)$ we denote its ordered eigenvalues
by $\lambda_{1}({\mathbb{A}})\le \lambda_{2}({\mathbb{A}})$
and by $\vv_{1}({\mathbb{A}}),\vv_{2}({\mathbb{A}})$ their corresponding normalized eigenvectors, which afterwards will be  denoted shortly with $\lambda_{1}, \lambda_{2}, \vv_{1}, \vv_{2}$  whenever there is no risk of confusion. \\
For every $\nu\neq 1$, 
$\boldsymbol{\xi}\in \mathbb R^{2}$ and $\mathbb{A}\!\in\! {\rm Sym}_{2,2}(\R)$\, we set
\begin{equation}\label{gA}
g_{\mathbb{A}}(\boldsymbol{\xi})=\left |\mathbb{A}+\boldsymbol{\xi}\otimes \boldsymbol{\xi}\right |^{2}+\dfrac{\nu}{(1-\nu)}\left ( {\rm Tr}\, \mathbb{A}+ |\boldsymbol{\xi}|^{2}\right )^{2}.
\end{equation}
\begin{lemma} \label{minconv} If $\nu\in (-1,1/2)$\,, then
\begin{equation} \label{(4.1)}\displaystyle\min_{\boldsymbol{\xi}\in \mathbb R^{2}} g_{\mathbb{A}}(\boldsymbol{\xi})=\left\{\begin{array}{ll}
\ g_{\mathbb A}(\mathbf{0}) \ \ &\ \text{ \rm if}\ \nu\lambda_{2}+\lambda_{1}\ge 0\\
&\\
\ (1+\nu)(\lambda_{2}(\mathbb{A}))^{2} &\ \text{ \rm if} \ \nu\lambda_{2}+\lambda_{1}< 0\,.
\end{array}
\right.
\end{equation}
\end{lemma}
\begin{proof} It is worth noticing that
minimum in \eqref{(4.1)} is achieved
since $g_{\mathbb{A}}\in C(\mathbb R^{2})$ and $g_{\mathbb{A}}(\boldsymbol{\xi})\to +\infty$ as $|\boldsymbol{\xi}|\to +\infty$. Let $\mathbb M\in O(2)$ be such that $\mathbb M^{T}\mathbb{A}\mathbb M=\text{ diag}(\lambda_{1},\lambda_{2})$. Then it is readily seen that by setting $x:=\boldsymbol{\xi}\cdot\vv_{1},\ \ y:=\boldsymbol{\xi}\cdot\vv_{2}$ we have
\begin{equation*}
\tilde g_{\mathbb{A}}(x,y)
\, := \
g_{\mathbb{A}}(\boldsymbol{\xi})
\ =\
( x^{2}+\lambda_{1})^{2}+( y^{2}+\lambda_{2})^{2}+2x^{2}y^{2}+\dfrac{\nu}{1-\nu}\left (\lambda_{1}+\lambda_{2}+x^{2}+y^{2})^{2}\right)
\end{equation*}
and an easy computation shows that if $\nu\lambda_{2}+\lambda_{1}\ge 0$ then minimum is attained at $(x,y)=(0,0)$.
Else, if $\nu\lambda_{2}+\lambda_{1}<0$ then
either $\nu\lambda_{1}+\lambda_{2}\ge 0$ or $\nu\lambda_{2}+\lambda_{1}\le \nu\lambda_{1}+\lambda_{2}<0$.\\
 In the first case $D \tilde g_{\mathbb{A}}(x,y)=(0,0)$ if and only if
$(x,y)\in \{(\pm\sqrt{-\nu\lambda_{2}-\lambda_{1}},0), (0,0)\}$
and $\tilde g_{\mathbb{A}}(x,y)= (1+\nu)\lambda_{2}^{2}$ or ${g}_{\mathbb{A}}(x,y)={g}_{\mathbb{A}}(0,0)>  (1+\nu)\lambda_{2}^{2}$; in the latter one $D \tilde{g}_{\mathbb{A}}(x,y)=(0,0)$ also at $(x_{*} ,\pm y_{*})=(0,\pm\sqrt{-\nu\lambda_{2}-\lambda_{1}})$ with $\tilde{g}_{\mathbb{A}}(x_{*},\pm y_{*})= (1+\nu)\lambda_{1}^{2}$. Hence $$\min_{\boldsymbol{\xi}\in \mathbb R^{2}} g_{\mathbb{A}}(\boldsymbol{\xi})=(1+\nu)\lambda_{2}^{2}$$ if $\nu\lambda_{2}+\lambda_{1}< 0\le \nu\lambda_{1}+\lambda_{2}$ and
$$\min_{\boldsymbol{\xi}\in \mathbb R^{2}} g_{\mathbb{A}}(\boldsymbol{\xi})=(1+\nu)\min\{\lambda_{2}^{2}, \lambda_{1}^{2}\}$$ if $\nu\lambda_{2}+\lambda_{1}\le \nu\lambda_{1}+\lambda_{2}<0$.
In the latter case if $\nu\in (-1,0)$ then $\lambda_{1}\le \lambda_{2}\le-\nu\lambda_{1}$, hence $\lambda_{1}\le \lambda_{2}\le0$ and $|\lambda_{1}|\ge |\lambda_{2}|$. If $\nu\in [0,1/2)$ then $\lambda_{1}< 0$ and either $\lambda_{2}> |\lambda_{1}|>0$ or $\lambda_{1}\le \lambda_{2}\le 0$.
In the first case we get
 necessarily $\nu> 0$ and  $|\lambda_{1}|>\nu^{-1}(1-\nu)\lambda_{2}> \lambda_{2}$, a contradiction. Therefore $|\lambda_{2}|\le |\lambda_{1}|$ and
 $$\min_{\boldsymbol{\xi}\in \mathbb R^{2}} g_{\mathbb{A}}(\boldsymbol{\xi})=(1+\nu)\lambda_{2}^{2}$$
 whenever $\nu\lambda_{2}+\lambda_{1}<0$  thus proving the thesis.
 \end{proof}
Lemma \ref{minconv} proves quite useful in the perspective of the next Proposition and the subsequent Examples, since the two alternatives in the right-hand side of \eqref{(4.1)}
correspond respectively to locally flat or oscillating equilibrium configurations.

\begin{proposition}
\label{flat}
   If $\v_*\in \argmin \mathcal{F}$ and the ordered  eigenvalues $\lambda_1 \leq \lambda_2$
   of $\mathbb{E}(\v_*)$ fulfil $\nu\lambda_2+\lambda_1\geq 0$ in the whole set $\Omega$, then
   \begin{equation}\label{gtilde}
     \widetilde{\mathcal{G}}_h(\v_*,\zeta)  \ \geq \ \widetilde{\mathcal{G}}_h(\v_*,0).
   \end{equation}
If in addition $\nu\lambda_2\!+\!\lambda_1 \!>\! 0$ in a set of positive measure,
then the inequality in \eqref{gtilde} is strict for every $\zeta\not\equiv 0$.
 \end{proposition}
 \begin{proof} Due to \eqref{(4.1)} in Lemma \ref{minconv}: $\nu\lambda_2+\lambda_1\geq 0$ entails $g_{2\mathbb{E}(\uu_*)}(\boldsymbol{\xi})\geq
 g_{2\mathbb{E}(\uu_*)}(\mathbf{0})$\,, moreover $\nu\lambda_2+\lambda_1> 0$ entails $g_{2\mathbb{E}(\uu_*)}(\boldsymbol{\xi})>
 g_{2\mathbb{E}(\uu_*)}(\mathbf{0})$ . Hence
\begin{equation*}
J\big({\mathbb D}(\v_*,\zeta)\big)\,=\, \frac {E}{8(1+\nu)}\,g_{2\mathbb{E}(\v_*)} (Dw)\, \geq \, J\big(\mathbb{E}(\v_*)\big)
\end{equation*}
and, for $\zeta\in \mathcal{A}^i$,
\begin{eqnarray*}
 \widetilde{\mathcal{G}}_{h}(\v_*,\zeta) &=&   h^{\alpha}F_{h}^{b}(\zeta)+h^{2\alpha+1}\!\int_{\Omega }\!\,J\big({\mathbb D}(\v_*,\zeta)\big)
    -h^{2\alpha+1}\!\!\int_{\partial \Omega}\!\! \mathbf{f}\cdot\v_*
     \geq \\
 &\geq& h^{\alpha}F_{h}^{b}(\zeta)+
     h^{2\alpha+1}\!\int_{\Omega }\!\,J\big({\mathbb E}(\v_*)\big)
     -h^{2\alpha+1}\!\!\int_{\partial \Omega}\!\! \mathbf{f}\cdot\v_*
          \\
&\geq&\widetilde {\mathcal{G}}_{h}(\v_*,0) \ .
   \end{eqnarray*}
Moreover the first inequality in the last computation is strict whenever
$\nu\lambda_2+\lambda_1 > 0$ in a set of positive measure.
 \end{proof}
\begin{remark}\label{rmk4.4}
\rm{Notice that $s_1:=\frac{E}{1-\nu^2}(\nu\lambda_2+\lambda_1)$ is the smallest eigenvalue of the stress tensor
$\mathbb{T}(\vv)=J'\big(\mathbb{E}(\vv)\big)$. Therefore Proposition \ref{flat} shows that,
if the eigenvalues of the stress tensor are both strictly positive almost everywhere, then we can expect only one flat minimizer ($\zeta\equiv 0$). On the other hand, the possible occurrence of oscillating configurations requires the presence of a compressive state on a region of positive measure: that is to say
the stress tensor must have at least one negative eigenvalue on set of positive measure.}
\end{remark}
We show  some examples clarifying how the asymptotic behavior of functionals $\widetilde {\mathcal G}_{h}$
provides useful
information about minimizers when $\Omega$ is an annular set.\\
Set  $0< R_{1} < R_{2},\ p_1,\,p_2\,\in\, \R,\ \ \Omega:=B_{R_{2}}\setminus B_{R_{1}}, $ and consider uniform in-plane normal traction/compression at each component of the boundary.
$${\bf f}= -p_{1}\dfrac{\xx}{R_{1}}{\bf 1}_{\{|\xx|=R_{1}\}}+p_{2}\dfrac{\xx}{R_{2}}{\bf 1}_{\{|\xx|=R_{2}\}}.$$
Therefore  $\v\in \arg\min \mathcal F_{1,0}$ entails
\begin{equation}\label{vv}
\vv(\xx)=(a+b|\xx|^{-2})\xx,
\end{equation}
and
exploiting polar coordinates $\xx=(r\cos\theta,r\sin\theta)$ we obtain
 \begin{equation*}
  \E(\vv)\ =\ \left[
  \begin{array}{cc}
     a-\dfrac{b}{r^{2}} \cos 2\theta& -\dfrac{b}{r^{2}} \sin 2\theta \  \\
     &\\
    -\dfrac{b}{r^{2}} \sin 2\theta &  a+\dfrac{b}{r^{2}} \cos 2\theta
  \end{array}
  \right]\ .
\end{equation*}
By using Neumann boundary condition $J'(\E(\v)){\bf n}={\bf f}$ on $\partial\Omega$, we get :

\begin{equation}\label{ab}
p_{i}=E(1+\nu)^{-1}(a(1+\nu)(1-\nu)^{-1}-bR_{i}^{-2}),\ \ \ i=1,2 \end{equation}
that is
\begin{equation*}
a=\dfrac{(1-\nu)(p_{2}R_{2}^{2}-p_{1}R_{1}^{2})}{E(R_{2}^{2}-R_{1}^{2})};\ \
b=\dfrac{(1+\nu)(p_{2}-p_{1})R_{1}^{2}R_{2}^{2}}{E(R_{2}^{2}-R_{1}^{2})}.
\end{equation*}
It is worth noticing that
$a- b{r^{-2}},\ a+b{r^{-2}}$  are the eigenvalues of $ \mathbb E(\vv)$
and $(\cos\theta,\sin\theta), \, (-\sin\theta,\cos\theta)$ the corresponding normalized eigenvectors $\forall r\!\in\! [R_{1}, R_{2}]$;
order may change according to $\sign(b)$\!.\\
We examine several different cases which may occur.
\begin{example}\label{rad}
 \textbf{Radially oscillating minimizers.}
{\rm Set $\Gamma=\partial\Omega,  \nu\in (-1,1/2)$, $i=0$ and either $p_{1}\leq p_{2}<0$ or $p_{2}\leq p_{1}<0$. In the first case we get
$b\ge 0$ in the second one $b\le 0$. However in both cases  $\nu\lambda_{2}+\lambda_{1}<0$ in the whole annular set.\\
Set also $\vv(\x)=(a+b|\xx|^{-2})\xx\in \arg\min \mathcal F_{0,1} $.\\
   Choose  $\sigma_{h}\to 0_{+},\ \beta_{h}\to +\infty$,
$\psi_{h}:\mathbb R\to \mathbb R\ \ (R_{2}-R_{1})\text{-periodic}$ such that
\begin{equation}
\psi_{h}(t)= \max\left\{0, \min \{t-R_{1}-\sigma_{h}, R_{2}-\sigma_{h}-t\} \right\}
\end{equation}
and set $\psi^{*}_{h}:= \psi_{h}*\rho_{h}$ being $\rho_{h}$  mollifiers such that $\spt\rho_{h}\subset [-\sigma_{h},\sigma_{h}]$. Then by denoting the floor of a real number (maximum integer not exceeding the number) with $\lfloor\cdot\rfloor$ and setting $r=|x|$\,,
\begin{equation*}\label{onda}\zeta_{h}(r)=\left\{\begin{array}{ll}&\lfloor\beta_{h}\rfloor^{-1}
\sqrt{2(1-\nu)br^{-2}-2a(\nu+1)}\,\psi_{h}^{*}(R_{1}+(r-R_{1})\lfloor\beta_{h}\rfloor) \ \ \hbox{if} \  p_{1}\le p_{2}<0\,,\\
&\\
&\lfloor\beta_{h}\rfloor^{-1}\sqrt{2(\nu-1)br^{-2}-2a(\nu+1)}\,\psi_{h}^{*}(R_{1}+(r-R_{1})\lfloor\beta_{h}\rfloor) \ \ \hbox{if} \ p_{2}\le p_{1}<0\,,
\end{array}\right.
\end{equation*} $\zeta_h':=\partial\zeta/\partial r$\,, \
$D\zeta_h\,=\,(\zeta_{h,1},\zeta_{h,2})\,=\,(x_1/r,x_2/r)\,\zeta_h'$ \
and
\begin{equation*}\mathbb M(\theta)\, =\, \left[
  \begin{array}{cc}
    \cos \theta& -\sin \theta \  \\
     &\\
    \sin \theta &  \cos \theta.
  \end{array}
  \right] \qquad \mathbb S(\theta)\, =\,
  \left[
  \begin{array}{cc}
    \cos^{2}\theta\ \ &  \  \sin\theta\cos\theta\\
    &\\
     \sin\theta\cos\theta\ \  &  \sin^{2}\theta
       \end{array}
  \right]\, =\,
     (\zeta_h')^{-2}\,D\zeta_h\otimes D\zeta_h\ .
\end{equation*}
So $\mathbb{M}^{T}\mathbb{S}\,\mathbb{M}=\mathbf{e}_1\otimes\mathbf{e}_1$ and there exists $\Omega_{h}\subset\Omega$ with $|\Omega_{h}|\sim \sigma_{h}$ such that,  $|(\psi_{h}^*)'|=1$ on $\Omega\setminus\Omega_{h}$. Then referring to \eqref{gA} and \eqref{vv},
for every  $x\in\Omega\setminus\Omega_{h} $ we have
\begin{equation*}\begin{array}{ll}
&\\
&g_{2\E(\v)}(D\zeta_{h})= \left | 2\E(\vv)+D\zeta_{h}\otimes D\zeta_{h}\right |^{2}+\dfrac{\nu}{1-\nu}\left |2 \,{\rm div}\,\vv+|D\zeta_{h}|^{2}\right |^{2}=\vspace{0.2cm}\\
&\left|  2\mathbb M^{T}E(\vv)\mathbb M+\mathbb M^{T}D\zeta_{h}\otimes D\zeta_{h}\mathbb M\right |^{2}+\dfrac{\nu}{1-\nu}\left |4a+|D\zeta_{h}|^{2}\right |^{2}=
\vspace{0.2cm}\\
& \left |2(a-br^{-2}){\bf e}_{1}\otimes{\bf e}_{1}+2(a+br^{-2}){\bf e}_{2}\otimes{\bf e}_{2} +|{\zeta}_{h}'|^{2}\mathbb M^{T}\mathbb S\,\mathbb M\right |^{2}+\dfrac{\nu}{1-\nu}\left |4a+|\zeta_{h}'|^{2}\right |^{2}=\vspace{0.2cm}\\
& (2a-2br^{-2}+|\zeta_{h}'|^{2})^{2}+4(a+br^{-2})^{2}+\dfrac{\nu}{1-\nu}\left |4a+|\zeta_{h}'|^{2}\right |^{2}\,.
\end{array}
\end{equation*}
If $p_{1}\le p_{2}<0$, we have $b\geq 0\,,$ $|\zeta_h'|^2=2(1-\nu)br^{-1}-2a(\nu+1)+O(\lfloor \beta_h\rfloor^{-2})$ on $\Omega\setminus\Omega_{h}$, hence
$$g_{2\E(\v)}(D\zeta_{h})\,=\,4(1+\nu)(a+br^{-2})^{2}+O(\lfloor\beta_{h}\rfloor ^{-1})\,,
$$
\begin{equation*}\begin{array}{ll}
&\displaystyle\int_{\Omega}I_{\vv}(\xx,D\zeta_{h})\,d\xx=
\int_{\Omega\setminus\Omega_{h}}I_{\vv}(\xx,D\zeta_{h})\,d\xx+
\int_{\Omega_{h}}I_{\vv}(\xx,D\zeta_{h})\,d\xx=
\vspace{0.2cm}
\\
&=\dfrac{E}{2(1-\nu)}\displaystyle\int_{\Omega\setminus\Omega_{h}}
\{(a+b|\xx|^{-2})^{2}+O(\beta_{h}^{-1})\}
\,d\xx+O(\sigma_{h})\to \dfrac{E}{2(1-\nu)}\int_{\Omega}(a+b|\xx|^{-2})^{2}\,dx\\
\end{array}
\end{equation*}
Analogously, if $p_{2}\le p_{1}<0$,
then $b\leq 0\,$ and $|\zeta_h'|^2=2(\nu-1)br^{-2}-2a(\nu+1)+O(\lfloor \beta_h\rfloor^{-1})$ on $\Omega\setminus\Omega_{h}$, hence
$$g_{2\E(\v)}(D\zeta_{h})\,=\,4(1+\nu)(a-br^{-2})^{2}+O(\lfloor\beta_{h}\rfloor ^{-1})\,,
$$
\begin{equation*}
\displaystyle\int_{\Omega}I_{\vv}(\xx,D\zeta_{h})\,d\xx\to \dfrac{E}{2(1-\nu)}\int_{\Omega}(a-b|\xx|^{-2})^{2}\,d\xx\,.
\end{equation*}
By Lemma \ref{minconv} we know that
\begin{equation*}\displaystyle \min_{\boldsymbol{\xi}\in\mathbb R^{2}}I_{\vv}(x,\boldsymbol{\xi})= \left\{\begin{array}{ll}& \dfrac{E}{2(1-\nu)}(a+{b}|\xx|^{-2})^{2}\ \hbox{if}\ p_{2}\le p_{1}<0\,,\\
&\\
&\dfrac{E}{2(1-\nu)}(a-{b}|\xx|^{-2})^{2}\ \hbox{if}\ p_{1}\le p_{2}<0\,,\
\end{array}\right.
\end{equation*}
therefore in both cases we have proved that
\begin{equation*}
\int_{\Omega}I_{\vv}(\xx,D\zeta_{h})\,dx\to \min\{\mathcal G^{**}(\vv,\zeta): \zeta\in H^{1,4}(\Omega),\ \zeta=0\ \hbox{in}\ \partial\Omega\}\ .\end{equation*}
Moreover \vskip-0.4cm
$$h^{-2\alpha-1}\widetilde {\mathcal G}(\vv,\zeta_{h})=h^{-\alpha-1}F_{h}^{b}(\zeta_{h})+\int_{\Omega}I_{\vv}(x,D\zeta_{h})\,d\xx-\int_{\partial \Omega} \mathbf{f}\cdot\v\, d\mathcal H^{1}\,,$$
$$h^{-\alpha-1}F_{h}^{b}(\zeta_{h})\sim h^{2-\alpha}\beta_{h}\sigma_{h}^{-1}\ .$$
Therefore by Theorem \ref{relax} for every choice of $\beta_{h},\ \sigma_{h}$ satisfying the conditions detailed before, $(\vv,{\zeta}_{h})$ can be viewed as an asymptotically minimizing sequence of $\widetilde{\mathcal G}_{h}$ whose out-of-plane component exhibits periodic oscillations (period: $\beta_{h}^{-1}$;
asymptotic amplitude: $\sqrt{2(1-\nu)br^{-2}-2a(\nu+1)}$ if $p_{1}\!\le\! p_{2}\!<\!0$
and $\sqrt{2(\nu-1)br^{-2}-2a(\nu+1)}$ if $p_{2}\!\le\! p_{1}\!<\!0$) in the radial direction in the whole annular set. The optimal choice of $\beta_{h}$ can be determined heuristically as follows: previous estimates show that
$$h^{-2\alpha-1}\widetilde {\mathcal G}(\vv,\zeta_{h})- \min{\mathcal G}^{**}=R_{h}$$
where $R_{h}\sim h^{2-\alpha}\beta_{h}\sigma_{h}^{-1}+ \beta_{h}^{-1}+\sigma_{h}$. So, 
approximatively, we have to minimize the last term. A direct calculation shows that
the best choice corresponds to  ${\beta_h}^{-1}\sim h^{2/3-\alpha/3}, \ \sigma_{h}\sim h^{5/3(2-\alpha)}.$}\vskip-0.2cm
\end{example}
\begin{example} \label{radflat}\textbf{Flat minimizer.}
{\rm
  Let $\,\Gamma=\partial\Omega,\  \nu\in [0,1/2),\ i=0$ or $i=1\,,\ p_{1}\ge 0$\,,\\ so that
 $R_{1}^{2}a\ge(1-2\nu)b$
and by Lemma \ref{minconv} we get\vskip-0.5cm
$$
\min\{\mathcal G^{**}(\vv,\zeta):\ \zeta\in H^{1,4}(\Omega),\ \zeta=0\ \hbox{in}\ \partial\Omega\}\ = \ \int_\Omega I_\vv(\xx,0)\,d\xx .$$
Obviously the minimum is attained at $\zeta\equiv 0$ that is we have a flat minimizer.}
\end{example}
\begin{remark}
{\rm Let $\Gamma=\partial\Omega,\  \nu\in (-1,1/2)\,,\ $i=0$\,,\  p_{1}< 0\le p_{2}$. Hence
$a>0,\ b> 0,\ \nu\lambda_{2}+\lambda_{1}=a-br^{-2}+\nu(a+br^{-2})\ge 0$ in the annular set $A_1=\{\overline R:=\sqrt{(1-\nu)(1+\nu)^{-1}{b}{a}^{-1}}\le r\le R_2\}$ and $< 0$ in the annular set $A_2=\{R_1\le r<\overline R\}$. Then by the same computations performed in previous examples we can build minimizers which are flat in $A_1$ and oscillating in $A_2$.}
 \end{remark}
 \begin{example} \label{extang}\textbf{Tangentially oscillating minimizers.}
{\rm
 Let $\Gamma=\partial B_{R_1},\ \nu\in (-1,1/2)$, $i=1$  and choose $p_1>0$, $p_2>0$ such that
$p_2R_2^2=p_1R_1^2$.
If $\v\in \arg\min \mathcal F_{1,0}$ we find again
$\vv(\xx)=(a+b|\xx|^{-2})\xx$
with
\begin{equation}\label{abvinyl}
a=0,\ \ \ b=-(1+\nu)E^{-1}p_{1}R_{1}^{2}< 0.\end{equation}
Hence
 $\lambda_{1}=b{r^{-2}}< 0 < - b{r^{-2}} =\lambda_{2}$ are the eigenvalues of $\mathbb E(\vv)$
and $\v_{1}= (-\sin\theta,\cos\theta),\ \v_{2}=(\cos\theta,\sin\theta)\,$ the corresponding normalized eigenvectors.
\\
Choose  $\sigma_{h}\to 0_{+},\ \beta_{h}\to +\infty$,
$\phi_{h}:\mathbb R\to \mathbb R,\ \ 2\pi\text{-periodic}$ defined by
\begin{equation}\phi_{h}(t)=\max\left\{0, \min \{t-\sigma_{h}, 2\pi-\sigma_{h}-t\}\right\}
\end{equation}
and set $\phi_{h}^{*}:=\phi_{h}*\rho_{h}$ being $\rho_{h}$  mollifiers such that $\spt\rho_{h}\subset [-\sigma_{h},\sigma_{h}]$. Let
\begin{equation*}
  \zeta_h(r,\theta)=\sqrt{-2b(1-\nu)}\,\,
  {\lfloor \beta_{h}\rfloor}^{-1}\phi_h^{*}\!\left(\lfloor \beta_{h}\rfloor\theta\right)\big(\delta_{h}^{-1}(r-R_{1}){\bf 1}_{[R_{1},R_{1}+\delta_{h}]}(r)+{\bf 1}_{[R_{1}+\delta_{h},R_{2}]}(r)\big)
\end{equation*}
with $\delta_{h}\to 0_{+},\  \beta_{h}^{-1}\delta_{h}^{-1}\to 0$.
Then there exists $\Omega_{h}\subset\Omega$ with $|\Omega_{h}|\sim \sigma_{h}$ such that for every  $x\in\Omega\setminus\Omega_{h} $ we have $|(\phi_{h}^*)'|=1$ on $\Omega\setminus\Omega_{h}$. Therefore referring to \eqref{gA} and  \eqref{vv}
and by setting
\begin{equation*}R_{*}(\theta)\ =\ \left[
  \begin{array}{cc}
   -\sin \theta&  \cos \theta \  \\
     &\\
   \cos \theta & \sin \theta
  \end{array}
  \right]
  \end{equation*}
we get
\begin{equation*}\begin{array}{ll}
& \displaystyle\int_{\Omega\setminus\Omega_{h}}\left(\left | 2E(\vv)+D\zeta_{h}\otimes D\zeta_{h}\right |^{2}+\dfrac{\nu}{1-\nu}\left |D\zeta_{h}\right |^{4}\right)\,d\xx=\\
&\\
&=\displaystyle\int_{\Omega\setminus\Omega_{h}}\left(\left|  2R_{*}^{T}E(\vv)R_{*}+R_{*}^{T}D\zeta_{h}\otimes D\zeta_{h}R_{*}\right |^{2}+\dfrac{\nu}{1-\nu}\left |D\zeta_{h}\right |^{4}\right)\,d\xx=\\
&\\
&= \displaystyle\int_{\Omega\setminus\Omega_{h}}4(1+\nu)b^{2}|\xx|^{-4}\,d\xx +O(\lfloor\beta_{h}\rfloor^{-1}\delta_{h}^{-1})+O(\sigma_{h})+O(\delta_{h}).
\end{array}
\end{equation*}
By using now Lemma \ref{minconv} and by arguing  as in Example \ref{rad} we get
\begin{equation*}
\int_{\Omega}I_{\vv}(\xx,D\zeta_{h})\,dx\,-\,\int_{\partial \Omega} \mathbf{f}\cdot\v\, d\mathcal H^{1}\to \min\{\mathcal G^{**}(\vv,\zeta): \zeta\in H^{1,4}(\Omega),\ \zeta=0\ \hbox{in}\ \partial\Omega\}\ .\end{equation*}
\begin{equation*}
h^{-2\alpha-1}\widetilde {\mathcal G}(\vv,\zeta_{h})\to \min\mathcal G^{**}=\dfrac{Eb^{2}}{2(1+\nu)}\int_{\Omega}|\xx|^{-4}\,d\xx
\,-\,\int_{\partial \Omega} \mathbf{f}\cdot\v\, d\mathcal H^{1}\,.
\end{equation*}
Moreover, since $h^{-\alpha-1}F_{h}^{b}(\zeta_{h})\sim h^{2-\alpha}\beta_{h}\sigma_{h}^{-1}$, we get
\begin{equation*}
\begin{array}{lll}
  h^{-2\alpha-1}\widetilde {\mathcal G}(\vv,\zeta_{h})&=&h^{-\alpha-1}F_{h}^{b}(\zeta_{h})+\int_{\Omega}I_{\vv}(x,D\zeta_{h})\,d\xx
-\int_{\partial \Omega} \mathbf{f}\cdot\v\, d\mathcal H^{1}\,=\,\vspace{0.25cm}\\
&=&
h^{2-\alpha}\beta_{h}\sigma_{h}^{-1}+
O(\lfloor\beta_{h}\rfloor^{-1}\delta_{h}^{-1})+O(\sigma_{h})+O(\delta_{h})\,.
\end{array}
\end{equation*}
Hence, here  the optimal choice is
${\beta_h}^{-1}\sim h^{1-\alpha/2},\ \delta_{h}\sim \beta_{h}^{-1/2},\ \sigma_{h}\sim h^{1-\alpha/2}\beta_{h}^{1/2}$.}
\end{example}
\begin{remark}\label{strategy} {\rm Thanks to Lemma \ref{minconv} and Proposition \ref{flat}, Examples \ref{rad}, \ref{radflat}, \ref{extang} constitute a paradigm for the construction of oscillating versus flat
approximated minimizers.\\ Moreover we sketch another technique to devise new ones, by this procedure: first  take a boundary force field,  construct the corresponding prestressed state (in $2$D there are a lot of
significant classical examples, see for instance those of Examples \ref{example 2.7},
\ref{example 2.8}) and look at the eigenvalues of the strain matrix:  it is not difficult to obtain examples according to either $\,\nu\lambda_{2}+\lambda_{1}\ge 0\,$ or $\,\nu\lambda_{2}+\lambda_{1}< 0\,$
in the whole plate.\\
In the first case through Lemma \ref{minconv} and Proposition \ref{flat} we argue that there is only a flat minimizer, in the second one a careful use of Lemma \ref{minconv}
on the pattern of Examples \ref{rad}, \ref{extang} allows an easy construction.
}\end{remark}






\begin{thebibliography}{99}

\bibitem{ABP}
G. Anzellotti, S. Baldo, D. Percivale, \textit{Dimension reduction in variational problems, asymptotic development
in $\Gamma$-convergence and thin structures in elasticity}. Asympt. Analysis, vol. 9, (1994) p.61-100, ISSN: 0921-7134.

\bibitem{ABG} M. Al-Gwaiz, V. Benci, G. Gazzola,
\textit{Bending and stretching energies in a rectangular plate modeling suspension bridges},
 Nonlinear Analysis 106, 18--34  (2014).


 \bibitem{AP} B. Audoly, Y. Pomeau, \textit{Elasticity and Geometry}, Oxford University Press, 2010.

\bibitem{BBGT}C. Baiocchi,  G. Buttazzo, F. Gastaldi, F. Tomarelli,
\textit{General existence results for unilateral problems in continuum mechanics}, Arch. Rational Mech. Anal. \textbf{100} (1988), 149--189.

\bibitem{BCDM} Ben Belgacem, S. Conti, A. DeSimone, S. M\"uller, \textit{Rigorous bounds for the Foppl-von
Karman theory of isotropically compressed plates}, J. Nonlinear Sci. 10 (2000), 661-683.


\bibitem{BCDM1} Ben Belgacem, S. Conti, A. DeSimone, S. M\"uller
\textit{Energy scaling of compressed elastic films}, Arch. Rat. Mech. Anal. 164 (2002), 1-37.



\bibitem{BEK} J. Bedrossian,  R. V. Kohn, \textit{Blister patterns and energy minimization in compressed thin
films on compliant substrates}, Comm. Pure Appl. Math. 68 (2015), 472-510.

\bibitem{BK} P. Bella and R. V. Kohn, \textit{Metric-induced wrinkling of a thin elastic sheet}, submitted to
Journal of Nonlinear Science, 2014.
\bibitem{BK1} P. Bella and R. V. Kohn, \textit{Wrinkles as the result of compressive stresses in an annular thin film}, Comm. Pure Appl. Math. 67 (2014), no. 5, 693-747.

\bibitem{BCDM} H. Ben Belgacem, S. Conti, A. DeSimone, S. M\"uller, \textit{Rigorous bounds for the F\"oppl-von
K\'arm\'an, theory of isotropically compressed plates},
J. Nonlinear Sci. 10 (2000), no. 6,
661-683.

\bibitem{BCM} D. P. Bourne, S. Conti, S. M\"uller
	\textit{Energy Bounds for a Compressed Elastic Film on a Substrate},  Accepted in Journal of Nonlinear Science.

\bibitem{BKN} J. Brandman, R. V. Kohn, H.-M. Nguyen,
\textit{Energy scaling laws for conically constrained
thin elastic sheets}, J. Elasticity 113 (2013), no. 2.


 \bibitem{BFF} K. Bhattacharya, I. Fonseca , G. Francfort,
\textit{ An asymptotic study of the debonding of thin films}, Arch. Ration. Mech. Anal. {\bf 161} (2002), no. 3, 205-229.

 \bibitem{BD} G.Buttazzo, G. Dal Maso, \textit{Singular perturbation problems in the calculus of variations}, Ann.Scuola Normale Sup., Cl.Sci., 4\,ser,, \textbf{11} 3 (1984) 395-430.

\bibitem{BT} G. Buttazzo, F. Tomarelli, \textit{Compatibility conditions for nonlinear Neumann problems}, Advances in Math. \textbf{89}
(1991), 127-143.

\bibitem{CLT}M. Carriero, A. Leaci, F. Tomarelli,
\textit{Strong solution for an elastic-plastic plate},
{Calc.Var.Partial Diff.Eq., }{\textbf{2}, }{2, (1994) 219-240.}

\bibitem{CM} Cerda, E., and Mahadevan, L. (2003).\textit{ Geometry and Physics of Wrinkling.} Phys. Rev. Lett. Vol. 90, N.7,
 2003,  1-5.

\bibitem{C} P. G. Ciarlet, Mathematical Elasticity, Volume II: Theory of Plates, Elsevier, 1997.

\bibitem{DMPN} G. Dal Maso, M. Negri, D. Percivale,  \textit{Linearized elasticity as $\Gamma$-limit of finite elasticity}, Set-Valued Anal. \textbf{10} (2002), no. 2-3, 165-183.

\bibitem{DPH} N. Damil, M. Potier-Ferry, Heng Hu, \textit{Membrane wrinkling revisited from a multi scale point of view}, Advanced Modeling Simulat. in Eng. Science, 16 (2014).


\bibitem{DSV} Davidovitch, B., Schroll, R.D., Vella, D., Adda-Bedia, M., Cerda, E.A. (2011).
\textit{Prototypical model for
tensional wrinkling in thin sheets}. Proc. Natl. Acad. Sci. U. S. A. 108, 18227-18232.



\bibitem{FJM} G. Frieseke, R.D. James, S. M\"uller, \textit{A Hierarky of Plate Models from Nonlinear Elasticity by Gamma-Convergence},
Arch. Rational Mech. Anal. 180 (2006), 183-236.


\bibitem{GO} G. Gioia, M. Ortiz, \textit{Delamination of compressed thin films}, Adv. Appl. Mech. 33 (1997), 119-192.

\bibitem{G} C. E. Gutti\'{e}rrez, The Monge-Amp\`ere Equation, Bikh\"auser, 2001.

\bibitem{HQC} T. J. Healey, Quingdu Li, R.B. Cheng, \textit{Wrinkling behavior of highly stretched rectangular elastic films via parametric global bifurcation}, J. Nonlinear Science 23 (2013), 777-805.

\bibitem{JS} R.L. Jerrard, P. Sternberg, \textit{Critical points via $\Gamma$-convergence: General theory and applications}, Jour. Eur. Math. Soc., vol. 11, no. 4, 705-753, 2009.


\bibitem{KN}V. Kohn and H.-M. Nguyen, \textit{Analysis of a compressed thin film bonded to a compliant substrate: the energy
scaling law}, J. Nonlinear Sci. 23 (2013), no. 3, 343-362.

\bibitem{LM} M. Lecumberry, S. M\"uller, \textit{Stability of slender bodies under compression and validity of
von K\'{a}rm\'{a}n theory}, Arch. Rational Mech. Anal. 193 (2009), 255-310.

\bibitem{LNMG} M. Leocmach, M. Nespoulous, S. Manneville, T. Gibaud, \textit{Hierarchical wrinkling in a
confined permeable biogel} Sci. Adv. 1, e1500608 (2015).

\bibitem{LMP} M. Lewicka, L. Mahadevan, M. Pakzad,
\textit{The Monge-Amp\`ere constraint: matching of isometries, density and regularity and elastic theories of shallow shells}, accepted in Annales de l'Institut Henri Poincare (C) Non Linear Analysis.

\bibitem{LP} M. Lewicka, R. Pakzad
\textit{Prestrained elasticity: from shape formation to Monge-Amp\`ere anomalies}, to appear in the Notices of the AMS (January 2016).

 \bibitem{LOP} M. Lewicka, P. Ochoa, R. Pakzad,
\textit{Variational models for prestrained plates with Monge-Amp\`ere constraint}, Diff. Integral Equations,
Vol. 28, no 9-10 (2015), 861-898.

\bibitem{MP} F. Maddalena, D. Percivale, \textit{Variational models for peeling problems},
 Interfaces Free Boundaries, \textbf{10} (2008), 503-516.

\bibitem{MPPT} F. Maddalena, D. Percivale, G.Puglisi, L. Truskinowsky,
\textit{Mechanics of reversible unzipping},
Continuum Mech. Thermodyn. (2009) 21:251268.
%

\bibitem{MPT} F. Maddalena, D. Percivale, F. Tomarelli, \textit{Adhesive flexible material structures}, Discr. Continuous Dynamic. Systems B, Vol. 17, Num. 2, March 2012, 553-574.

\bibitem{MPT1} F. Maddalena, D. Percivale, F. Tomarelli, \textit{Elastic structures in adhesion interaction},  Variational Analysis and Aerospace Engineering , Editors: A.Frediani, G.Buttazzo,  Ser. Springer  Optimization and its Applications, Vol. 66, ISBN 978-1-4614-2434-5, (2012) 289-304.

\bibitem{MPT2} F. Maddalena, D. Percivale, F. Tomarelli, \textit{Local and nonlocal energies in adhesive interaction, }{IMA Journal of Applied Mathematics (2016) 81 {\bf 6}, 1051--1075.
    }

\bibitem{Olb}{H. Olbermann, }{\textit Energy scaling law for a single disclination in a thin elastic sheet, }{Arxiv Preprint, Sept. 24, 2015 }
%
\bibitem{PTplate}{D. Percivale, F. Tomarelli, }{{\textit From SBD to SBH: the elastic-plastic plate, }}
{Interfaces Free Boundaries, }{{\textbf 4}, 2,}{ (2002) 137--165.}

\bibitem{PT2} {D. Percivale, F. Tomarelli, }{\textit{A variational principle for plastic hinges in a beam},}{  Math. Models Methods Appl. Sci., 19, n.12, (2009), 2263-2297.}


\bibitem{PDF}
E. Puntel, L. Deseri, E. Fried
\textit{Wrinkling of a Stretched Thin Sheet}, Journ. of Elasticity
November 2011, Volume 105, Issue 1, pp 137-170.

\bibitem{RT}{J. Rauch, B.A. Taylor, }{\textit{The Dirichelet problem for the multidimensional Monge-Amp\`{e}re equation, }}{Rocky Mountain Journ. of Math.
Vol. 7, N. 2, 345-364  (1977).}





 \end{thebibliography}
\end{document}